\documentclass{amsart}
\usepackage{amsmath, amssymb,amsthm}
\usepackage{mathrsfs}
\usepackage{enumitem}
\usepackage{url}

\usepackage[pdftex]{graphicx}
 
\newcommand{\D}{\mathbb{D}}
\newcommand{\C}{\mathbb{C}}
\newcommand{\R}{\mathbb{R}}

\newcommand{\N}{\mathbb{N}}

\DeclareMathOperator{\sgn}{sgn}
\DeclareMathOperator{\Tr}{Tr}

\newcommand{\mc}[1]{\mathcal{#1}}

\theoremstyle{plain}
\newtheorem{theorem}{Theorem}
\newtheorem{corollary}{Corollary}
\newtheorem{proposition}{Proposition}
\newtheorem{lemma}{Lemma}

\theoremstyle{definition}

\theoremstyle{remark}

\title[H\"ormander-Bernhardsson constant]{A Painlev\'e equation for the  H\"ormander-Bernhardsson constant}
\author{Friedrich Littmann}
\address{Department of Mathematics, North Dakota State University, Fargo, ND 58103}
\email{Friedrich.Littmann@ndsu.edu}


\begin{document}

\begin{abstract} The H\"ormander-Bernhardsson constant $\mathscr{C}$ is the sharp constant in $|f(0)|\le \mathscr{C} \|f\|_1$ for entire functions of exponential type $\le \pi$. We prove that $\mathscr{C} = 2\pi \theta_*^{-2}$ where $\theta_*$ is the least positive singularity of  the solution  $v$ of the cosh-Gordon equation $v_{\theta\theta} + v_\theta/\theta = \cosh(v)$ that is regular at the origin with $v(0)=0$ . 

It is known that $\mathscr{C}$ is a scaling limit in $n$ from the analogous problem for polynomials of degree $\le n$. We reformulate the polynomial problem  as a Pad\'e approximation problem at infinity. The associated matrix Riemann-Hilbert problem is analyzed by a Deift-Zhou steepest descent whose local parametrix is built from a Painlev\'e transcendent. 
\end{abstract}

\maketitle

\section{Introduction and Result}

Let $\tau>0$, and $0\le p\le \infty$.  We denote by $PW^p_\tau $ the Paley-Wiener space of entire functions of exponential type $\tau>0$ in $L^p(\R)$. Consider the problem 
\begin{align}\label{PWC-extremal}
\mathfrak{C}_\tau = \inf_{\substack{F\in PW_\tau^1\\ F(0)=1}}\|F\|_1
\end{align}

Hörmander and Bernhardsson introduced this problem in \cite{BH93} in connection with an extension of Bohr's inequality. They proved the extremizer has only real simple zeros, and computed $0.5409288219 \le \mathscr{C} \le  0.5409288220$ where $\mathscr{C} =  1/\mathfrak{C}_\pi$.

Recent works   \cite{ BORS26-2, BORS26} give several new properties of the extremal function. In  \cite{BORS26-2}  a third order linear differential equation for the extremal function is derived, several open questions from \cite{BH93} are settled, and algorithms are developed that allow computing  a decimal approximation to  $\mathscr{C}$ precise to 100 digits. Recent preprints  \cite{Gor26} and independently \cite{GRR26}  focus on higher dimensional versions of this problem.

Let $n\in\N$, and define $\mc{P}_n$  to be the vector space of polynomials of degree at most $n$. The corresponding problem for $\mc{P}_n$ is
\begin{align}\label{Cn-extremal}
C_n = \inf_{\substack{P\in \mc{P}_{n} \\ P(0)=1}} \|P\|_1.
\end{align}

The polynomial problem was considered initially in \cite{AmZie76} where properties of the extremizer are investigated. In \cite{LevLub15}, the Paley-Wiener and the polynomial  setting are considered. In particular, scaling limits are developed that allow transferring polynomial limits into the Paley-Wiener setting. For our purpose, we need  
\begin{align}\label{LL-scaling}
\lim_{n\to \infty} n C_{n} = \pi\mathfrak{C}_\pi
\end{align}
which follows from \cite[Theorem 1.3]{LevLub15}.  For additional generalizations and applications we refer to \cite{BCOS22, CMS19, DGT21, Gor05}. 

To describe the main result of the present paper, let $v = v(\theta)$ be the  solution of the cosh-Gordon Painlev\'e III equation
\[
v_{\theta\theta} +\frac1\theta v_\theta = \cosh(v)
\]
that is regular near the origin (hence $v'(0)=0$) and satisfies $v(0) =0$. Denote by $\theta_*$ the smallest positive value of $\theta$ for which $v$ has a singularity. 

\begin{theorem}\label{theorem1} The limit relation
\[
\lim_{n\to \infty} n C_{n}  = \frac{\theta_*^2}{2} 
\]
holds.
\end{theorem}

It follows from \eqref{LL-scaling} that
\[
\mathscr{C} = \frac{2\pi }{\theta_*^{2}}
\]

A decimal approximation to $\theta_*$ is $\theta_* = 3.4081591997747264290903593760707477132636 \hdots $
which gives
\[
 \frac{2\pi}{\theta_*^2} =0.54092882190183058939288205899969038685\hdots,
\]
in agreement with the approximation to $\mathscr{C}$ from \cite{BORS26-2}.  For the constant $L_\tau(1)$ from \cite[(1.9)]{BORS26-2} we  prove the following result.\footnote{While developing a Python routine to compute the decimal approximation of $\theta_*$, Claude Opus  decided (unprompted) to also compute the local expansion of $v$ at $\theta_*$, apply the PSLQ algorithm to $\mathscr{C}, L_\tau(1)$, and $c_2$, and conjectured their relation. }  

\begin{theorem}\label{theorem2}
 Let the local expansion of $v$ at $\theta_*$ be given by
\begin{align}\label{localv-attheta*}
v(\theta) = \log\frac{4}{(\theta - \theta_*)^2}  - \frac{1}{\theta_*}(\theta-\theta_*) + c_2 (\theta-\theta_*)^2+\hdots
\end{align}

Then
\[
c_2 = \frac{\mathscr{C}}{8\pi} +\frac{L_\tau(1)}{3\pi}
\]
\end{theorem}

Behind both results are representations of the extremal functions for \eqref{PWC-extremal} and \eqref{Cn-extremal} in terms of the solution of a Riemann-Hilbert matrix problem for the Cosh-Gordon equation from \cite[Chapter 13]{FIKN06}. We describe the full relations in Section \ref{extremal-convergence-section}.
\medskip

We give a brief outline of the paper.  Section \ref{OPUC-section} applies a method of Peherstorfer  \cite{Peh88} to translate sign change patterns satisfying equation \eqref{trig-condition} below for $1\le \nu\le n$ into OPUC language. The main idea  is to keep $c$ arbitrary and develop rational functions with positive real part in the unit disk that parametrize the sign change patterns for every $c$ and trace the location of zeros and poles. At the critical $c$, equation \eqref{trig-condition} holds for $\nu = n+1$ as well.

Section \ref{Pade-section} reformulates this into the $[n/n]$ Pad\'e approximation problem at infinity for the function
\[
z\mapsto \left(\frac{z-1}{z+1}\right)^{1/2} e^{-c/(2z)}.
\]
and a Chen-Ismail  ladder operator \cite{CI97} is utilized to derive a first order ODE for the zero flow of the Pad\'e denominator $Q_{k,c}$. This is used to show that $c=C_n$ corresponds to an initial origin crossing of a zero, and as a consequence to the least positive zero of $c\mapsto Q_{k,c}(0)$.

 Section \ref{RHP-section} formulates the non-hermitian orthogonality of the Pad\'e denominator as a matrix Riemann-Hilbert problem. Setting $c = d/n$, Deift-Zhou steepest descent \cite{DZ93} leads to a local parametrix at the origin that identifies the least positive zero of $d\mapsto Q_{k,d/k}(0)$ as the singularity of a cosh-Gordon equation. Section \ref{constant-asymptotics-section} gives the proof of Theorem \ref{theorem1}. Section \ref{extremal-convergence-section} derives the convergence of the RHP solutions to the extremal function for \eqref{PWC-extremal}, and Section \ref{Heun-ODE-section} contains the relation between the coefficients of the Heun differential equations of the local parametrix and the constants $\mathscr{C}$ and $L_\tau(1)$ that leads to the proof of Theorem \ref{theorem2}.

\medskip

\subsection*{Generative AI Statement} The author acknowledges extensive chats with Claude Opus about Chen-Ismail ladder operators and Deift-Zhou steepest descent. The article was fully written by the author. Claude Opus was used to double check the algebra and to aid with images and decimal approximations.

\section{Sign changes and paraorthogonal polynomials}\label{OPUC-section}

We recall the following characterization for $C_n$, cf.\ \cite[eq.\ (8)]{AmZie76}. If there exists real-valued $P_0\in \mc{P}_n$ with $P_0(0)=1$ and 
\[
\int_{-1}^1 x^\nu {\rm sgn}(P_0(x))dx =0, \qquad (\nu=1,\hdots,n)
\]
then  for any $P \in\mc{P}_n$ with $P(0) =1$,  writing $P = 1 + xG$ with $G\in \mc{P}_{n-1}$,  
\[
\|P\|_1 \ge  \left| \int_{-1}^1 (1+x G(x)) {\rm sgn}(P_0(x)) dx \right| = \left| \int_{-1}^1 {\rm sgn}(P_0(x)) dx\right|
\] 

In the other direction, a compactness argument  shows that an extremizer $P_0$ exists and if ${\rm sgn}(P_0)$ is not orthogonal to $x,\hdots,x^n$, then $P_1 = P_0 + xG$ with $G\in \mc{P}_{n-1}$ can be constructed with $\|P_1\|_1<\|P_0\|_1$.

Following \cite{Peh88}, we replace $x^\nu$ by the Chebyshev polynomials $U_\nu$ of the second kind, and we use $c$ for $\nu=0$, i.e., we seek $P_0$ so that
\begin{align}\label{poly-signs}
\begin{split}
\int_{-1}^1  U_\nu(x)  \sgn P_0(x) dx = c U_\nu(0)\qquad (\nu=0,\hdots,n)
\end{split}
\end{align}

Defining $H$ by  $H(\theta) = \sgn P_0(\cos \theta)$ for $0<\theta<\pi$, we have $H(0+)= (-1)^k$.  Since $U_\nu(x) = \sin((\nu+1)\theta)/\sin\theta$ where $x = \cos\theta$, we seek the smallest positive $c$ for which a $\pm1$-function $H$ on $(0,\pi)$ with $n$ sign changes exists which satisfies
\begin{align}\label{trig-condition}
\int_0^\pi H(\theta) \sin(\nu \theta) d\theta = cU_{\nu-1}(0), \qquad (\nu = 1,\hdots,n+1)
\end{align}

Let $0<\theta_1<\hdots<\theta_n<\pi$ and let $H$ be the $\pm 1$ function on $(0,\pi)$ with $H(0+)= \pm1$ changing sign exactly at $\theta_1,\hdots,\theta_n$. Then for $\ell\ge 1$, a direct calculation gives 
\begin{align}\label{sign-evaluation}
\int_0^\pi \sin(\ell\theta) H(\theta) d\theta =\pm  \frac1\ell \left( 1 +2\sum_{i=1}^n (-1)^i \cos(\ell\theta_i) - (-1)^{n+\ell}\right)
\end{align}
 
Let $n=2k$. For $\varphi_j, \psi_j\in (0,\pi)$ such that
\begin{align} \label{Phi-interlace}
\begin{split}
&0<\psi_1<\varphi_1<\hdots<\varphi_k<\pi \qquad(k\text{ odd}), \\
 &0<\varphi_1<\psi_1<\hdots<\psi_k<\pi \qquad (k\text{ even})
\end{split}
\end{align}
we define
\begin{align}\label{Dn-def}
D_n(z) = - \frac{(z+(-1)^k) \prod_{j=1}^k (1- 2\cos(\varphi_j) z + z^2 )}{(z-(-1)^k)\prod_{j=1}^k (1-2\cos(\psi_j) z + z^2)}
\end{align}

\begin{proposition} Let $n=2k$ and let $\varphi_j, \psi_j\in (0,\pi)$ satisfy \eqref{Phi-interlace}. Define $b_\ell$, $\ell\ge 1$ by
\begin{align}\label{log-coefficients}
\log D_n(z) =- \sum_{\ell=1}^\infty b_\ell z^\ell,
\end{align}
and let $H_n = {\rm sgn}(i D_n(e^{i\theta}))$. Then $H_n$ is a $\pm 1$ function on $(0,\pi)$ changing sign exactly at $\varphi_j$ and $\psi_j$ with $H_n(0+)=(-1)^{k+1}$ and
\begin{align}\label{bell}
b_\ell = \int_0^\pi \sin(\ell \theta) H_n(\theta) d\theta \qquad (\ell = 1,\hdots, n)
\end{align}
\end{proposition}

\begin{proof} We observe that $D_n$ away from the poles is purely imaginary on the unit circle. The sign changes and the normalization at $\theta =0$ are a direct calculation. We note
\[
\log\frac{1\pm z}{1\mp z} = \pm \sum_{\ell=1}^\infty \left(1-(-1)^\ell\right) \frac{z^\ell}{\ell}
\]
and 
\[
\ln(1-2\cos(\theta)z+z^2) = -2\sum_{\ell=1}^\infty \frac{\cos(\ell \theta)}{\ell} z^\ell.
\]

This implies 
\[
b_\ell =\frac{1}{\ell} \left( (-1)^{k+1} (1 - (-1)^\ell) + 2\sum_{j=1}^k \big(\cos(\ell \varphi_j) - \cos(\ell \psi_j)\big)\right)
\]
and \eqref{sign-evaluation} completes the proof.
\end{proof}

\noindent{\it Remark.} The argument is taken from \cite[(2)-(6) and (11)]{Peh88} where it is proved for every $n$,  and $b_\ell$ not depending on $n$. (Peherstorfer normalizes $H$ by $H(0+) =1$.) 

\medskip

We associate with a probability measure $\mu$ on the unit circle an analytic function $D$, the Herglotz function of $\mu$, defined for $z\in\D$ by
\[
D(z) = \frac{1}{2\pi}\int_0^{2\pi} \frac{e^{i\theta} +z}{e^{i\theta} -z} d\mu(\theta)
\]

We define by $\Phi_n$ for $n=0,1,...$ the monic orthogonal polynomial of degree $n$ of $\mu$. For a polynomial of degree $\le n$ we define $P^*(z) = z^n \overline{P(1/\bar{z})}$. Orthogonal polynomials satisfy a difference equation
\[
\Phi_{n+1}(z) = z\Phi_n(z) -\bar{\alpha}_n \Phi_n^*(z)
\]
where $\alpha_n$ are called the Verblunsky coefficients.

A partial fraction decomposition of $D_n$ in \eqref{Dn-def} shows that this function is a linear combination of Herglotz functions for the measures $\mu_j = \delta_{e^{\pm i\psi_j}}$, $j=1,...,k$ and $\mu_0 = \delta_{(-1)^k}$, and the interlacing assumption implies that the coefficients in this linear combination are all positive. Since $D_n(0)=1$, the function $D_n$ is the Herglotz function of a measure $\mu_n$ with these $2k+1$ mass points. Its Verblunsky coefficients satisfy $|\alpha_j(D_n)|<1$ for $j<n$ and $|\alpha_n(D_n)|=1$, and $\Phi_{n+1}(\mu_n,z)$ is the monic polynomial in the denominator of $D_n$ (cf.\ \cite[Theorem 2.2.12]{Sim05}). Comparing constant terms gives $\alpha_n(D_n) = (-1)^{k}$. 

Writing $\Psi_{n+1}(z) = \Phi_{n+1}((-\alpha_k);z)$ for the polynomial of the second kind, \cite[(3.2.33)]{Sim05} gives
\[
D_n(z)+ \frac{\Psi_{n+1}(z)}{\Phi_{n+1}(z)} =\mc{O}(z^{n+1})
\]

Comparing degrees and limit at infinity shows that the error term vanishes identically, i.e., $D_n$ is a quotient of paraorthogonal polynomials\footnote{Interlacing of the zeros of paraorthogonal polynomials was proved in \cite{Sim07} and independently in \cite{Wong07}, although a version for real Verblunsky coefficients  is already in \cite[p. 243 ff.]{Peh88}.}.

 Let $F$ be analytic in the unit disk with $F(0)=1$ with Taylor coefficients $2 d_k(F)$ for $k\ge 1$. We define
\[
\Lambda(F) =  \sup\{n\in\N_0: \exists \textrm{  Herglotz function $D$ with $d_k(F)=d_k(D)$ for $k\le n$}\}.
\]

 Since
\[
 \sum_{k=1}^\infty U_{k-1}(0) z^k = \frac{z}{1+z^2},
\]
equations \eqref{trig-condition} and \eqref{log-coefficients} lead us to define $F_c$ by
\[
F_c(z) = \exp\left( -\frac{cz}{1+z^2} \right)
\]

We have $\Lambda(F)\in \N_0\cup\{\infty\}$. We define $\alpha_k(F)$  for $k< \Lambda(F)$ by 
\[
\alpha_k(F) = \alpha_k\left(D_{\Lambda(F)}\right)
\]
 and we note that $|\alpha_k(F)|< 1$ for $k\le  \Lambda(F)-2$. The Verblunsky theorem \cite[Section 3.2]{Sim05} implies that $\alpha_k(F)$ is well defined.
 If $F = F_c$, we write $d_k(c)$, $\alpha_k(c)$ and $\Lambda(c)$. 

\begin{proposition}\label{Prop2} Let $n =2k\in \N$  and $c$ real. The following are equivalent.
 \begin{enumerate}
\item $c = C_n$.
\item $\Lambda(c)\ge n$, $|\alpha_\ell(c)|<1$ for $\ell<n$ and $\alpha_n(c)\in \{\pm 1\}$.
\end{enumerate}
\end{proposition}

\begin{proof}  
If $c = C_n$, there exists $H$ with $n$ sign changes satisfying \eqref{trig-condition} for $\nu = 1,...,n+1$. The function $D_n$ in \eqref{Dn-def} satisfies $|\alpha_j(D_n)|<1 $ for $j<n$ and $|\alpha_n(D_n)|=1$. Since \eqref{trig-condition} holds for $\nu\le n+1$, we have
\[
D_n(z) = \exp\left( -c\sum_{\nu=1}^{n+1}  U_{\nu-1}(0) z^\nu \right) + \mc{O}(z^{n+2})
\]
and this gives $\alpha_j(D_n) = \alpha_j(c)$ for $j\le n$.

For the second claim, assume first $\alpha_n(c) = (-1)^k$. Defining 
\[
\Phi_{n+1}(z) = \Phi_{n+1}(\alpha_0(c),\hdots,\alpha_n(c),(-1)^k;z),
\]
 the function $D_n = -\Psi_{n+1}/\Phi_{n+1}$ has the form \eqref{Dn-def}, hence \eqref{bell} gives \eqref{trig-condition} for $\ell = 1,...,n$. Since $\alpha_n(D_n) = \alpha_n(c) = (-1)^k$, \cite[(3.2.33)]{Sim05} for index $n+1$ gives
\[
D_n(z) = \exp\left( -\sum_{\nu=1}^{n+1} c U_{\nu-1}(0) z^\nu \right) +\mc{O}(z^{n+2})
\]
which is  \eqref{trig-condition}  for $\ell = n+1$. For $\alpha_n(c) = (-1)^{k+1}$ , under $z\mapsto -z$, the Verblunsky coefficients pick up a sign $(-1)^{j+1}$. Since $F_{-c}(z) = F_c(-z)$, for even $n$ we have $\alpha_n(-c) = -\alpha_n(c)$.
\end{proof}

\begin{corollary}\label{Cor2} Let $n=2k\in \N$ with $k\ge 2$. Assume $\varphi_j=\varphi_j(c)$ and $\psi_j=\psi_j(c)$, $1\le j\le k$ are continuous in $c$ such that
\begin{align}\label{bn-fixing}
D_{2k}(z) = \exp\left(-\frac{cz}{1+z^2}\right) + \mc{O}(z^{2k+1})
\end{align}
holds. Define $I_k = \{c\in\R:\text{ condition \eqref{Phi-interlace} holds}\}$, let $I_{k,0}$ be the connected component  of $I_k$ that contains $0$, and let $H=H_c$ be the corresponding $\pm1$ function on $(0,\pi)$ with $H_c(0+)=(-1)^{k+1}$. If $\sup I_{k,0}<\infty$, then $ C_{2k-2} = \sup I_{k,0}$. 
\end{corollary}

\begin{proof} We note that for $c=0$, 
\[
D_{2k}(z)  =\frac{1 + (-1)^k z^{2k+1}}{1-(-1)^k z^{2k+1}}
\]
hence $0\in I_k$ and $I_{k,0}\neq \emptyset$. If $c\in I_{k,0}$, then \eqref{Phi-interlace} implies that $|\alpha_{2k-2}(D_{2k})|<1$, and \eqref{bn-fixing} implies that $\alpha_{2k-2}(D_{2k}) = \alpha_{2k-2}(c)$. Hence $C_{2k-2}\ge \sup I_{k,0}$. 

Since the sign change locations of $H_c$ are assumed to be continuous functions in $c$, strict interlacing can fail by two adjacent sign changes merging (in this case, $H_{\sup I_{k,0}}$ has $\le 2k-2$ sign changes), or by a sign change reaching $\theta=0$ or $\theta=\pi$. In the latter case, the symmetry for the extremal associated with \eqref{Cn-extremal} implies that sign changes reach both endpoints, so  number of sign changes of $H_{\sup I_{k,0}}$ is $\le 2k-2$ as well. Relation \eqref{bn-fixing} implies \eqref{trig-condition} for $\ell = 1,\hdots , 2k$ and that $c$ is contained in the interval $I_{k,0}$, and by continuity, this remains true for $c = \sup I_{k,0}$. It follows that  $C_{2k-2} = \sup I_{k,0}$.
\end{proof}
 
Heuristically, one expects that the sign change count drops by $2$ and not by $4$, so merging is expected to happen either at the origin, or at the endpoints. It turns out that this is indeed the case; with our normalization, the origin merging event happens at $c = C_{2k-2}$ and the endpoint event at $c = -C_{2k-2}$.

\section{Pad\'e Approximations and zero dynamics}\label{Pade-section}

Writing $D_{2k}$ in \eqref{Dn-def} in the form
\[
D_{2k}(z)= \frac{(-1)^k+z}{(-1)^k-z} \cdot \frac{\prod_{j=1}^k (z+\frac1z - 2\cos(\varphi_j))}{\prod_{j=1}^k(z+\frac1z - 2\cos(\psi_j))}
\]
there are two ways how to map this relation from the unit circle to the real line with the substitution $x = \frac12 (z+ z^{-1}) = \cos\theta$. First, by direct application, which is the change from \eqref{poly-signs} to \eqref{trig-condition}, and second by first multiplying by $((-1)^k - z)/((-1)^k + z)$ before substituting $x$. Assuming that \eqref{bn-fixing} holds, from $D_{2k}(1/z) = -D_{2k}(z)$  and   $F_c(z) = F_c(1/z)$ it follows that
\begin{align}\label{bn-fixing-2}
D_{2k}(z) = F_c(z)  +\mc{O}(z^{-2k-1})
\end{align}

Viewing \eqref{bn-fixing} and \eqref{bn-fixing-2} as a two-point Pad\'e approximation at $(0,\infty)$, the term  $((-1)^k+z)/((-1)^k-z)$ introduces fixed side constraints at $z = \pm 1$. By instead considering the Pad\'e approximation  in $x=\frac12(z+z^{-1})$   to  $F_c(z) [(1+z)/(1-z)]^{(-1)^{k+1}}$, we obtain a Pad\'e approximation at infinity  without side constraints.

\begin{figure}[ht]
  \centering
  \includegraphics[width=\textwidth]{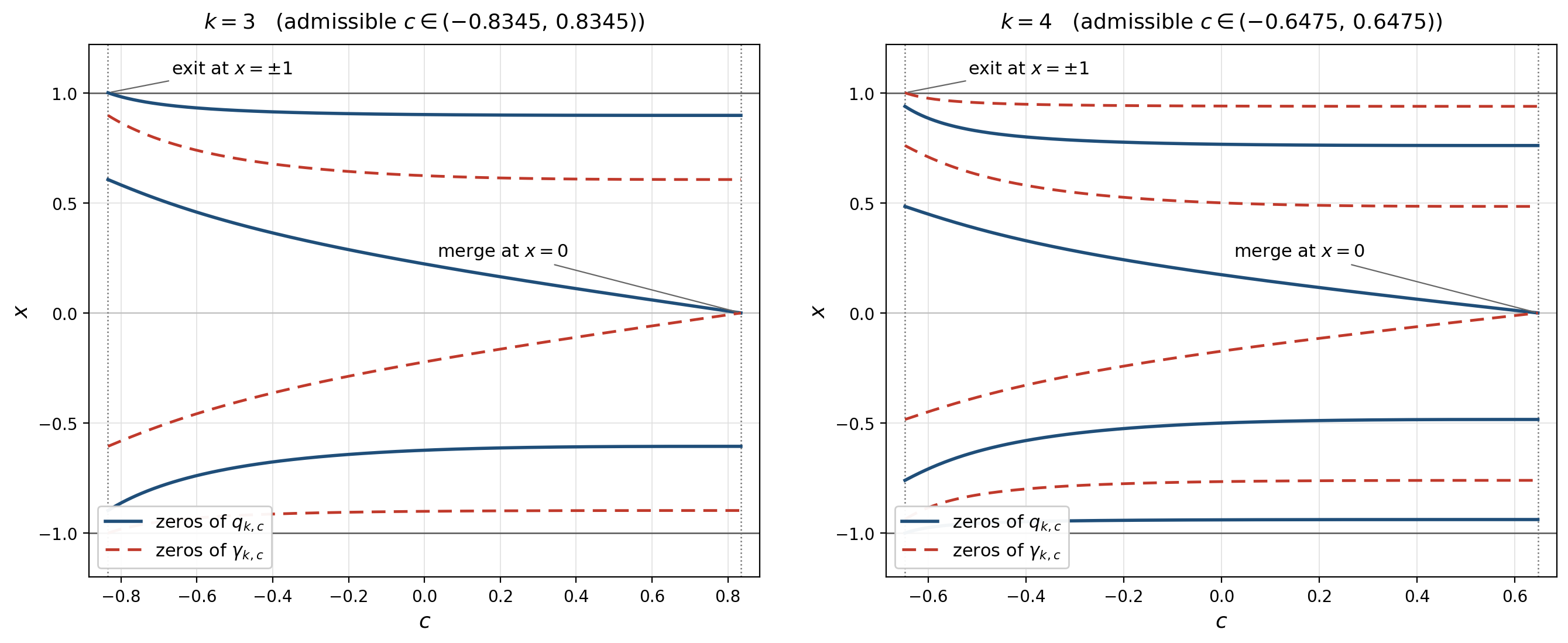}
  \caption{Zeros of $q_{k,c}$ (solid) and $\gamma_{k,c}$ (dashed) as functions of $c$
    for $k=3, 4$, showing the exit at $x=\pm1$ and the merge at $x=0$.}
  \label{fig:pqzeros}
\end{figure}

Define $J_k$ to be the set of real $c$ so that 
there exist polynomials $\gamma_{k,c}, q_{k,c}$ of exact degree $k$, $q_{k,c}$ monic, having simple zeros in $(-1,1)$ that strictly interlace,
\begin{align}\label{Pade-definition-qkc}
\begin{split}
&\sqrt{\frac{x-1}{x+1}}e^{-c/(2x)} - \frac{\gamma_{k,c}(x)}{q_{k,c}(x)} = \mc{O}(x^{-2k-1}),\qquad k\textrm{ even}\\
&\sqrt{\frac{x+1}{x-1}}e^{-c/(2x)} - \frac{\gamma_{k,c}(x)}{q_{k,c}(x)} = \mc{O}(x^{-2k-1}),\qquad k\textrm { odd}
\end{split}
\end{align}
and let $J_{k,0}$ be the connected component of $J_k$ containing $0$. (For $c=0$, the polynomials can be explicitly computed from orthogonal polynomial theory for  Jacobi weights, and a direct check gives $0\in J_k$.)

For the remaining part of the paper we consider $k$ even. We define
\[
W_{c}(x) = \frac{1}{2i} \left(\sqrt{\frac{x-1}{x+1}}\right) e^{-c/(2x)}
\]
with branch on $[-1,1]$ and $2iW_c$ positive for real $x$ with $|x|>1$, so that $W_c$ is analytic in $\C\backslash[-1,1]$ and has limit $1/(2i)$ at infinity. Let $\Gamma_R$ be the circle with center at the origin and radius $R>1$ oriented {\it clockwise}. Define the functional
\[
\mc{L}_c[f] =  \int_{\Gamma_R} f(y) W_c(y) dy
\]

 Clockwise orientation and the factor $1/(2i)$ for the weight are chosen so that for $c=0$, the functional $\mc{L}_c$ is positive, and that  the jump $(W_c)_+ - (W_c)_-$ in  \eqref{omegac-def} equals $\omega_c$.

  Denote by $x_\ell(c) $, $\ell = 1,...,k$ the zeros of $q_{k,c}$, ordered by their magnitude
\[
0<x_{1,k}(c) < -x_{2,k}(c) < x_{3,k}(c) <\hdots < -x_{k,k}(c)
\]
and we write $x_\ell$ if the degree is unambiguous. At $c=0$ their values are explicit since 
\begin{align}\label{qk0-is-Pkpm12}
2^k q_{k,0}(x) = P_k^{(1/2,-1/2)}(x)
\end{align}

 Writing the zeros as $\cos [(2\pi m)/(2k+1)]$ (cf.\ \cite[18.16.2]{NIST}), setting $\cos\phi = \sin(\pi/2 -\phi)$, the reordered zeros are
\begin{align}\label{xell-zero}
x_{\ell}(0) = (-1)^{\ell+1} \sin\left(\frac{(2\ell-1)\pi}{4k+2}\right)
\end{align}

The arguments in the proofs of Lemma \ref{positive-weight}, \ref{Pade-basics}, and \ref{lemma4} are standard (e.g., \cite[Chapter 5.6]{BGM96}), but we include them here since the usual chain of arguments derives interlacing as a consequence.

\begin{lemma}\label{positive-weight} Let $k$ be even and $c\in J_{k,0}$. Define $\lambda_{\ell,c}$ by
\begin{align}\label{lambda-quadrature-weight}
\lambda_{\ell,c} = -\frac{\pi \gamma_{k,c}(x_\ell(c))}{q_{k,c}'(x_\ell(c))}
\end{align}
Then $\lambda_{\ell,c}>0$ for $1\le \ell\le k$. Defining $\nu_c = \sum \lambda_{\ell, c} \delta_{x_\ell(c)}$, 
\[
\mc{L}_c[p] = \int pd\nu_c, \qquad (p \in \mc{P}_{2k-1})
\]
the measure $\nu_c$ is positive, and the largest zero in the interlacing belongs to $\gamma_{k,c}$. 
\end{lemma}

\begin{proof} Since $q_{k,c}$ is monic and has all of its zeros in $(-1,1)$, it is real valued on the real line. From $q_{k,c} (2iW_c) - \gamma_{k,c} = \mc{O}(y^{-k-1})$, and $q_{k,c} (2iW_c) = y^k +\mc{O}(y^{k-1})$, it follows that $\gamma_{k,c}$ is also monic (with zeros in $(-1,1)$) and hence real valued.  The residue theorem gives
\[
\lambda_{\ell,c} = \mc{L}_c\left[ \frac{q_{k,c}(y)}{(y-x_\ell(c)) q_{k,c}'(x_\ell(c))}\right]
\]

Let $p\in \mc{P}_{2k-1}$ and write $p = q_{k,c} s + r$ where $\deg s, r\le k-1$, and
\[
r(y) = \sum_{\ell =1}^k  p(x_\ell(c)) \frac{q_{k,c}(y)}{(y-x_\ell(c)) q_{k,c}'(x_\ell(c))}
\]

We have $\mc{L}_c[q_{k,c} y^j]=0$ for $j\le k-1$, and it  follows that $\mc{L}_c$ restricted to $\mc{P}_{2k-1}$ is integration against $\nu_c$.  Furthermore,
\[
\lambda_{\ell,c} = \frac{1}{2i q_{k,c}'(x_\ell(c))} \int_{\Gamma_R} \frac{\gamma_{k,c}(y)}{y - x_\ell(c)} dy = -\frac{\pi \gamma_{k,c}(x_\ell(c))}{q_{k,c}'(x_\ell(c))}
\]

Since $2iW_c(y) = 1 - (1 + \frac{c}{2})y^{-1} +\hdots$, 
\[
\sum_{\ell =1}^k \lambda_{\ell, c} = \mc{L}_c[1] = \pi \left(1 + \frac{c}{2}\right)
\]
and hence  $-2\notin J_{k,0}$.  Strict interlacing implies that all $\lambda_{\ell,c}$ have the same sign. Since $\mc{L}_c[1]>0$ for $c>-2$, that sign must be positive for $c\in J_{k,0}$ and all $\lambda_{\ell,c}$. Strict interlacing gives two choices for the sign of $ \gamma_{k,c}(x_\ell(c))/ q_{k,c}'(x_\ell(c))$, and investigating both shows that $\lambda_{\ell,c}>0$  requires the largest zero to belong to $\gamma_{k,c}$. 
\end{proof}

\begin{lemma}\label{Jk0=Ik0} $J_{k,0} = I_{k,0}$.
\end{lemma} 

\begin{proof}  Recall $k$ even and $c\in J_{k,0}$.  With $x = \frac12(z+z^{-1})$ and $x\notin[-1,1]$ a direct calculation gives $x^2-1 = \frac14(z-z^{-1})^2$ and
\[
\sqrt{\frac{x-1}{x+1}} = \sqrt{\frac{x^2-1}{(x+1)^2}} =\sqrt{\frac{(z-1)^2}{(z+1)^2}}
\]

The right-most expression simplifies to $(1-z)/(1+z)$ in $|z|<1$ (and to $(z-1)/(z+1)$ in $|z|>1$). We obtain
\begin{align*}
 \frac{(1+z)\gamma_{k,c}(\frac12(z+z^{-1}))}{(1-z) q_{k,c}(\frac12(z+z^{-1}))} &= \exp\left(-\frac{cz}{1+z^2}\right) +\mc{O}(z^{2k+1}) \qquad (|z|<1) 
\end{align*}

We next use the sign change pattern. The interlacing assumption and the fact that the zeros are in $(-1,1)$ and that the largest zero belongs to $\gamma_{k,c}$ implies after multiplication by $z^k$ in numerator and denominator that they  satisfy the assumptions of Corollary \ref{Cor2}, hence $c\in I_{k,0}$.  The reverse argument is analogous.
\end{proof}

\begin{lemma}\label{Pade-basics}  The zeros of $q_{k,c}$ are analytic functions of $c$ on $J_{k,0}$
\end{lemma}

\begin{proof} Let $k$ be even. Expand for $|x|>1$
\[
2iW_c(x)=\sum_{m=0}^\infty w_m(c) x^{-m}
\]
and observe that $w_m$ are polynomials in $c$ that are real-valued for real $c$.  Set $\rho_{i+1}(c) = \mc{L}_c[x^i] = -\pi w_{i+1}(c)$. Since $\mc{L}_c[x^j q_{k,c}]=0$ for $j<k$ and $q_{k,c}$ monic, the system
\[
\sum_{i=0}^{k-1} b_i \rho_{i+j+1}(c) =- \rho_{k+j+1}(c)
\]
has a solution for $0\le j\le k-1$. Lemma \ref{positive-weight} implies that $H_k(c) = (\rho_{i+j+1})_{0\le i,j\le k-1}$ is non-singular, hence the unique $b_i$ are  the coefficients of $q_{k,c}$. Since $\gamma_{k,c}$ is the polynomial part of $q_{k,c}(2iW_c)$, its coefficients are polynomial in $b_i$ and $w_m(c)$. 

Since $\rho_m$ and $\det H_k$ are entire in $c$ and $\det H_k$ is nonzero on a neighborhood of $c$, the coefficients of both polynomial are holomorphic in $c$ by Cramer's rule. The implicit function theorem and assumption of simplicity of the zeros implies that the zeros are analytic in $c$.
\end{proof}

We note that $\sup J_{k,0}<\infty$ since $C_{2k-2}<\infty$. There is also a self-contained proof that the supremum is finite;  for large $c$, the zeros of the Pad\'e approximation are not in $(-1,1)$ since the influence from the algebraic term diminishes. The approximations converge to Pad\'e approximations of $e^{-c/(2x)}$ and some of the zeros become either large or non-real.

Since we need the Pad\'e approximations for $j\le k$ for the ladder operators, we define $P_{j,c}$ and monic $Q_{j,c}$ by
\begin{align}\label{Pade-definition}
Q_{j,c}(y) (2i W_c(y)) - P_{j,c}(y) = \mc{O}(y^{-j-1})
\end{align}

We observe  
\begin{align}\label{PQ-relation}
P_{j,c}(x) = (-1)^j Q_{j,c}(-x)
\end{align}
 and note that for even $k$, we have  $Q_{k,c} = q_{k,c}$.  For $f_c(y)$, we write $\dot{f}$ for the partial with respect to the parameter $c$, and $f'$ for the partial with respect to $y$.

 Define 
\[
h_j(c) = \mc{L}_c[Q_{j,c}^2]
\]

\begin{lemma}\label{lemma4} Let $k$ be even and $0\le j\le k-1$.
\begin{enumerate}[label=(\roman*)]
\item\label{one} The functions $h_j$ are analytic and positive on $J_{k,0}$.
\item\label{two} For $c\in J_{k,0}$, the polynomials $Q_{j,c}$ and $P_{j,c}$ have real coefficients.
\item\label{three} $Q_{k,c}(0)\neq 0$ for $c\in J_{k,0}$.
\end{enumerate}
\end{lemma}

\begin{proof}  The $Q_{j,c}$ is the monic orthogonal polynomial of degree $j$ for $\nu_c$. This implies positivity of $h_j$, and with \eqref{PQ-relation}    that $Q_{j,c}$ and $P_{j,c}$ have real coefficients. Analyticity follows  with an argument analogous to Lemma \ref{Pade-basics}.   Strict interlacing and \eqref{PQ-relation}  give \ref{three}.
\end{proof}
 
As a consequence, the Christoffel kernel $K_{k-1, c}$ given by
\[
K_{k-1,c}(x,y) =\sum_{j=0}^{k-1}  h_j(c)^{-1} Q_{j,c}(x) Q_{j,c}(y)
\]
is well defined and satisfies the Christoffel-Darboux identity
\[
K_{k-1,c}(x,y) = \frac{1}{h_{k-1}(c)} \frac{Q_{k,c}(x) Q_{k-1,c}(y) - Q_{k-1,c}(x) Q_{k,c}(y)}{x-y}
\]

\begin{lemma}\label{Q-residue-values} Let $k$ be even, $c\in J_{k,0}$, and $m\ge 1$. Then
\begin{align*}
\mc{L}_c[Q_{k,c}/y^m] & = -(-1)^{m-1}\pi  \frac{Q_{k,c}^{(m-1)}(0) }{(m-1)!}\\
\mc{L}_c[Q_{k,c}^2/y^{2m}]&=0\\
\mc{L}_c[Q_{k,c}^2/(y^2-1)]&=0\\
\mc{L}_c[yQ_{k,c}^2/(y^2-1)]&=-\pi  Q_{k,c}(1) Q_{k,c}(-1)
\end{align*}
\end{lemma}

\begin{proof}  Replacing $Q_{k,c}(y) W_c(y) = P_{k,c}(y) + \mc{O}(y^{-k-1})$ in the integral, evaluating the integral of the first term using the residue theorem and minding the orientation gives the first identity.  Likewise, for $\ell\ge 1$
\begin{align*}
\mc{L}_c[Q_{k,c}^2/y^\ell] &= \int_{|y|=R} \frac{Q_{k,c}(y) (P_{k,c}(y) + \mc{O}(|y|^{-k-1}))}{y^\ell} dy
\end{align*}

Hence the integral is the $(\ell-1)$st coefficient of the product $-\pi  Q_{k,c}(y) Q_{k,c}(-y)$, and for odd $\ell-1$ the coefficient is zero. Expanding $(y^2-1)^{-1}$ in powers of $y^{-2}$ gives the other claims. 
\end{proof}

Define $V_c(y) = -\log W_c$ and observe
\[
V_c'(y) = -\frac{1}{y^2-1}-\frac{c}{2y^2}
\]

We require analogues of  Toda deformation equation for $\dot{h}_k$ and Freud-type equation for $h_k$. We refer to \cite{CI97} and \cite[Chapter 3]{Ism05}, but emphasize that the equations in the following proposition are constrained by $c\in J_{k,0}$, and taking $k\to \infty$ requires $c\le d_k/k$ with bounded $d_k$. 

\begin{proposition} Let $k$ be even and $c\in J_{k,0}$. Then
\begin{align}\label{hdot-rep}
\begin{split}
\dot{h}_k(c) &= \frac{\pi}{2}   Q_{k,c}(0)^2\\
(2k+1) h_k(c) &= \pi Q_{k,c}(1) Q_{k,c}(-1) +\frac{\pi c}{2} Q_{k,c}(0)^2
\end{split}
\end{align}

Since $h_k(0)>0$, we have $h_{k}> 0$ on $J_{k,0}\cap [0,\infty)$.
\end{proposition}

\begin{proof}
Since $Q_{k,c}$ is monic, $\dot{h}_k = -\frac12 \mc{L}_c[Q_{k,c}^2/y]$, orthogonality with Lemma \ref{Q-residue-values} gives $\mc{L}_c[Q_{k,c}^2/y] = Q_{k,c}(0) \mc{L}_c[Q_{k,c}/y] =-\pi Q_{k,c}(0)^2$.

For the second identity, 
\[
\left( y Q_{k,c}^2 W_c\right)' = 2y Q_{k,c} Q_{k,c}' W_c + y Q_{k,c}^2 W_c' + Q_{k,c}^2 W_c
\]

We use $W_c' = -V_c' W_c$ to write $y Q_{k,c}^2 W_c' = - y Q_{k,c}^2 V_c' W_c$. Integrating on $\Gamma_R$ gives
\[
0 = \mc{L}_c[ 2y Q_{k,c} Q_{k,c}'] - \mc{L}_c[y Q_{k,c}^2 V_c'] + \mc{L}_c[Q_{k,c}^2 ]
\]

Since $Q_{k,c}$ is monic, $2y Q_{k,c}'(y) = 2k y^k + \hdots$, and orthogonality gives
\[
\mc{L}_c[y Q_{k,c}^2 V_c']  = (2k+1) h_{k}(c)
\]

The explicit form of $V_c'$ gives
\[
\mc{L}_c[y Q_{k,c}^2 V_c']  = -\mc{L}_c\left[\frac{y}{y^2 -1} Q_{k,c}^2\right] - \frac{c}{2} \mc{L}_c[ Q_{k,c}^2/y]
\]
and Lemma \ref{Q-residue-values} gives the explicit evaluations. 
\end{proof}

 Differentiation of $Q_{k,c}(x_\ell(c)) =0$  in $c$ gives
\begin{align}\label{xdot-quotient}
\dot{x}_\ell(c) = - \frac{\dot{Q}_{k,c}(x_\ell(c))}{Q_{k,c}'(x_\ell(c))}
\end{align}
and we aim to express the right side in terms of the zeros of $Q_{k,c}$.  For the numerator, from $\mc{L}_c[\dot{Q}_{k,c} Q_{j,c}] = \frac12 Q_{j,c}(0) \mc{L}_c[Q_{k,c}/y]$ and $k$ even, 
\begin{align}\label{Qdot-rep}
\begin{split}
\dot{Q}_{k,c}(y) &= -\frac{\pi}{2}  Q_{k,c}(0)\sum_{j=0}^{k-1} h_j^{-1} Q_{j,c}(0) Q_{j,c}(y) \\
&= -\frac{\pi}{2}  Q_{k,c}(0) K_{k-1,c}(0,y)
\end{split}
\end{align}

For the denominator of $\dot{x}_\ell$ we need to set up a Chen-Ismail type ladder operator since this will allow replacing $Q_{k,c}'$ by $Q_{k-1,c}$ at zeros of $Q_{k,c}$.

Define  $\mc{V}_c$ by
\[
\mc{V}_c(x,y) = \frac{V_c'(x) - V_c'(y)}{x-y}
\]

\begin{lemma} For $k$ even and $c\in J_{k,0}$,
\begin{align}\label{Qprime-numerator}
\begin{split}
 \mc{L}_c[\mc{V}_c(x,.) Q_{k,c}^2]  & = \frac{-(2k+1)h_k(c) x^2 +   c \dot{h}_k(c)}{x^2(x^2-1)}
\end{split}
\end{align}
\end{lemma} 

\begin{proof} We suppress the subscript $c$. A direct calculation gives
\[
\mc{V}(x,y) = (x+y)\left( \frac{1}{(x^2-1)(y^2-1)} +\frac{c}{2x^2y^2}\right)
\]

It follows with the evaluations above that
\[
 \mc{L}_c[\mc{V}(x,y) Q_k(y)^2] = \frac{\left(  \mc{L}_c[yQ_k^2/(y^2-1)]\right)}{x^2-1} +\frac{c \left( \mc{L}_c[Q_k^2/y]\right)}{2x^2} 
\]
and $ \mc{L}_c[Q_k^2/y] = -\pi Q_{k,c}(0)^2 = -2\dot{h}_k$ as in \eqref{hdot-rep}. Since
\[
\frac{A}{x^2 -1 } + \frac{B}{x^2} = \frac{A+B}{x^2} + \frac{A}{x^4(1-x^{-2})} = \frac{(A+B)x^2 - B}{x^2(x^2-1)},
\]
we expand in powers of $x^{-1}$ to evaluate $A+B$. For $y$ with $|y| = R$ 
\[
\mc{V}(x,y) = -\frac{V'(y)}{x} - \frac{y V'(y)}{x^2} + \mc{O}(x^{-3})
\]

Since $W' = -V'W$, we have $(Q_{k}^2 W)' = (2Q_{k}'Q_{k}  - V'Q_{k}^2) W$. When integrating on the circle, the term on the left is zero since it has a single-valued antiderivative in $|y|\ge 1$, and the first term on the right is zero due to orthogonality. It follows that  $\mc{L}[V'Q_k^2] =0$.

The identity $\mc{L}[yV'Q_k^2] = (2k+1)h_k$ from the proof of \eqref{hdot-rep}  implies the claim.
\end{proof}

\begin{lemma} Let $k$ be even. For $c\in J_{k,0}$
\begin{align}\label{Qprime-rep}
Q_{k,c}'(x) = -\frac{\mc{L}_c[\mc{V}_c(x,.) Q_{k,c}Q_{k-1,c}]}{h_{k-1}(c)} Q_{k,c}(x) +\frac{   \mc{L}_c[\mc{V}_c(x,.) Q_{k,c}^2] }{h_{k-1}(c)} Q_{k-1,c}(x)
\end{align}
\end{lemma}

\begin{proof} It follows from $(Q_{k} Q_{j} W)' = (Q_{k}'Q_{j} + Q_{k} Q_{j}' - V'Q_{k} Q_{j}) W$ that $\mc{L}_c[Q_k' Q_j] = \mc{L}_c[ V'Q_{k} Q_{j}],$ hence the expansion of $Q_k'$ with respect to $(Q_{j,c})$ is given by
\[
Q_{k}'(x)= \sum_{j=0}^{k-1} \frac{1}{h_j} \mc{L}_c[V' Q_{k} Q_{j}] Q_{j}(x) = \mc{L}_c \left[ V' Q_k K_{k-1}(x,.)\right]
\]

Using $V'(y) = V'(x) - (x-y) \mc{V}(x,y)$, observing that $V'(x) \mc{L}_c[Q_j Q_k]=0$ and applying the Christoffel-Darboux formula to  $(x-y) K_{k-1}(x,y)$ gives the claim.
\end{proof}

\begin{proposition} Let $k$ be even, $c\in J_{k,0}$, $c\ge 0$. Each function $\zeta =\pm  x_\ell$ satisfies the differential equation
\[
\dot{\zeta} = F(\zeta, c)
\]
with 
\[
F(\zeta,c) = \frac{\beta_k(c)}{c} \frac{\zeta (\zeta^2-1)}{\zeta^2 - \beta_k(c)}
\]
and
\begin{align}\label{betak-two-forms}
\beta_k(c) = \frac{c\dot{h}_k(c)}{(2k+1) h_k(c)} = \frac{ cQ_{k,c}(0)^2}{2Q_{k,c}(1) Q_{k,c}(-1) + cQ_{k,c}(0)^2}
\end{align}
\end{proposition}

\begin{proof}
Inserting the formulas for $\dot{Q}_k$ from \eqref{Qdot-rep} and $Q_k'$ from \eqref{Qprime-rep} and \eqref{Qprime-numerator} into \eqref{xdot-quotient}, and using \eqref{hdot-rep} gives the ODE with $\beta_k$ in the first form. The second form follows from \eqref{hdot-rep}. Finally, the ODE is unchanged when replacing $x_\ell$ by $-x_\ell$. 
\end{proof}

 Since $k$ is even, $Q_{k,c}(1) Q_{k,c}(-1)>0$ on $J_{k,0}$, hence $0<\beta_k<1$ for $c>0$ with $\beta_k(0)=0$. A direct calculation gives for $c>0$ and real $\zeta$
\[
\frac{d}{d\zeta} F(\zeta,c) = \frac{\beta_k(c)}{c}  \left( 1 + \frac{(1-\beta_k(c)) (\zeta^2 +\beta_k(c))}{(\zeta^2 - \beta_k(c))^2}\right)\ge \frac{\beta_k(c)}{c}>0
\]

Let $\zeta_1(c) <\zeta_{2}(c)$ with the property that no pole of $F$ lies between these two values. Then
\begin{align}\label{distance-flow}
\dot{\zeta}_2 - \dot{\zeta}_1 = F(\zeta_2) - F(\zeta_1) \ge 0
\end{align}
so the gap between $\zeta_1$ and $\zeta_2$ as a function of $c$ increases with $c$.  For the zeros $x_\ell$ this has the following consequences. The non-decreasing gaps and telescoping from \eqref{xell-zero} give for $1\le \ell\le k$ 
\[
|x_{\ell}(c)|\ge 2\cos \frac{\ell \pi}{4k+2} \sin \frac{(\ell-1)\pi}{4k+2}
\]

Estimating the cosine factor $\ge 1/\sqrt{2}$ and the sine factor $\ge (2\ell-2)/(4k+2)$
\begin{align}\label{square-zero-bound}
\sum_{\ell = 2}^\infty \frac{1}{ |x_{\ell}(c)|^2} \le 2 \sum_{\ell=2}^\infty \frac{k^2}{(\ell-1)^2} =\frac{\pi^2 k^2}{3}
\end{align}

We also need an upper bound on consecutive gaps. Define $\gamma_m(c) = |x_{(m+1)}(c)|-|x_{(m)}(c)|$. Since $|x_{(k)}(c)| - |x_{(k)}(0)|<0$ and $-|x_{(c)}(0)|\le 0$, 
\[
\sum_{m=1}^{k-1} (\gamma_m(c) - \gamma_m(0))\le |x_{(1)}(0)|
\]
giving the bound
\begin{align}\label{gap-bound}
\gamma_m(c)  \le \gamma_m(0)) +  |x_{(1)}(0)| \le \frac{\pi}{k}
\end{align}

\begin{lemma} Let $k$ be even. For $\ell\neq +$, 
\begin{align}\label{xdot-bound}
|\dot{x}_\ell| \le \frac{1}{2 k (|x_\ell| - x_+)}
\end{align}
\end{lemma}

\begin{proof} The ODE gives
\[
\dot{x}_\ell = F(x_\ell, c),
\]
and the bounds $|x_\ell^2-1|\le 1$ and the previous lemma imply
\[
|\dot{x}_\ell|\le \frac{1}{2k} \frac{|x_\ell|}{x_\ell^2 - \beta_k}
\]

We have $x_\ell^2 -\beta_k \ge x_\ell^2 - x_+^2 \ge (|x_\ell| - x_+) |x_\ell|$. 
\end{proof}

We return to the investigation of the flow of the zeros near $\sup J_{k,0}$. The following lemma shows that for $c\in J_{k,0}$, none of the zeros reach the pole of the ODE. The proof of \eqref{Rk-full-bounds} below shows in addition that $x_1(c)$ and $\beta_k(c)$ reach the origin simultaneously and with the same convergence rate at $c =\sup J_{k,0}$.

\begin{lemma} For $k$ even and $c\in J_{k,0}$, the inequality $x_\ell(c)^2 > \beta_k(c)$ holds.
\end{lemma} 

\begin{proof} If $x_\ell^2(c) = \beta_k(c)$, then $\pm \dot{x}_\ell(c) = \infty$ from $0<\beta_k(c)<1$, contradicting analyticity on $J_{k,0}$. For $c=0$ the inequality $x_\ell(0)^2>\beta_k(0) =0$ holds since $Q_{k,0}(0)\neq 0$. 
\end{proof}

\begin{proposition}\label{zero-flow-prop} Let $k$ be even. Then $(-1)^{k/2} Q_{k,c}(0) >0$ for $c\in J_{k,0}\cap[0,\infty)$ and
\[
\lim_{c\to \sup J_{k,0}} Q_{k,c}(0)=0,\qquad \lim_{c\to \sup J_{k,0}} Q_{k,c}(\pm 1)\neq 0
\]
\end{proposition}

\begin{proof} Denoting by $x_+(c)$ the non-negative zero of $Q_{k,c}$ closest to the origin, we have $\pm x_+(c) \notin [-\sqrt{\beta_k(c)},\sqrt{\beta_k(c)}]$. For any values $\pm x_\ell(c)$ that have the same sign, the distance between them is non-decreasing on $J_{k,0}$ by \eqref{distance-flow}, and the same argument applied to positive $\pm x_\ell(c)$  and $1$ (negative $\pm x_\ell(c)$ and $-1$) shows that the zeros maintain a positive distance to $\pm 1$ at $c = \sup J_{k,0}$. 

Since $\sup J_{k,0}<\infty$, collision at $x=0$ of $\pm x_\ell$ is the only remaining possibility of violating the conditions defining $J_{k,0}$. This implies that $Q_{k,c}(0)$ has the same sign as $Q_{k,0}(0)$.
\end{proof}

So far, for $q_{k,c}$ from \eqref{Pade-definition-qkc} with even $k$, Lemma \ref{Jk0=Ik0}  and Corollary \ref{Cor2} give  $\sup J_{k,0} = \sup I_{k,0} = C_{2k-2}$, and Proposition \ref{zero-flow-prop} implies that $\sup J_{k,0}$ is  the first positive zero of $c\mapsto q_{k,c}(0)$. In the next section, we set up a Riemann-Hilbert problem to compute 
\[
\lim_{k\to \infty}  q_{k,d/k}(0)
\]
for positive $d$. The Riemann-Hilbert problem turns out to have a singularity at $d = \lim_{k\to\infty} k (\sup J_{k,0})$. The next results provide estimates for remainders that are uniform up to the singularity.

\begin{lemma} Let $k$ be even.  For $c\le \sup J_{k,0}$
\begin{align}\label{Qpm1-bounds}
2k+1 \le 4^k Q_{k,c}(1) Q_{k,c}(-1) < 2k + 1 + \frac{\pi k}2 \sup J_{k,0}
\end{align}
\end{lemma}

\begin{proof} Differentiate the second equation in \eqref{hdot-rep} in $c$ and replace $\dot{h}_k$ to get
\[
\partial_c (Q_{k,c}(1) Q_{k,c}(-1)) = \pi k Q_{k,c}(0)^2 -\frac\pi2 \partial_c (Q_{k,c}(0)^2)
\]
From \eqref{Qdot-rep}, $\partial_c (Q_{k,c}(0))^2<0$, so $4^k Q_{k,c}(1) Q_{k,c}(-1)\ge 4^k Q_{k,0}(1) Q_{k,0}(-1) = 2k+1$, where the last value comes from \eqref{qk0-is-Pkpm12}, \cite[18.3.1, 18.6.1]{NIST}, and $k$ even. Integrating the derivative  on $0\le c \le \sup J_{k,0}$, multiplication by $4^k$, observing that $4^k Q_{k,c}(0)^2$ is decreasing in $c$ by \eqref{Qdot-rep}, and has value $1$ for $c=0$ gives the upper bound. 
\end{proof}

The main dynamics in $c$ is in $x_1(c)$. The convergence statements in the later sections require estimates for the zeros $x_\ell$ with $\ell\ge 2$. Let
\[
R_k(c) = \prod_{\ell=2}^{k} \left|x_{\ell,k}(c)\right|
\]

\begin{proposition} Let $k$ be even. Then $R_k$ is monotonically decreasing in $c$ and for $0\le c\le \sup J_{k,0}$ 
\begin{align}\label{Rk-full-bounds}
(2k+1)^2 \ge  4^k R_k^2(c) \ge \frac{4k+2}{\sup J_{k,0} }\
\end{align}
\end{proposition}

\begin{proof} Since each $|x_\ell|$ is monotonically decreasing, we have 
\[
(2k+1)^2\ge  4^k R_k^2(0)\ge  4^k R_k^2(c),
\]
where the upper bound comes from  $2^k Q_{k,0}(0)=1$ and \eqref{xell-zero}. We compute next the limit of the expression on the right as $c\to \sup J_{k,0}$. 

We note that the weight $\lambda_1$ of $x_1$ from \eqref{lambda-quadrature-weight} can be computed in two ways\footnote{During numerics, Claude Opus pointed out that $\lambda_{1}\to 0$ provides  information about $x_1$ that the ODE alone cannot give.}, from the quotient $-\pi \gamma_{k,c}/Q_{k,c}'$, and from the reciprocal of the Christoffel-Darboux kernel. We get
\[
-\pi\frac{Q_{k,c}(-x_1)}{Q_{k,c}'(x_1)}= \frac{1}{K_{k-1,c}(x_1,x_1)} = \frac{h_{k-1}(c)}{Q_{k,c}'(x_1) Q_{k-1,c}(x_1)}
\]
which gives $Q_{k-1,c}(x_1)/h_{k-1} = -(\pi Q_{k,c}(-x_1))^{-1}$. Inserting this into \eqref{xdot-quotient} with \eqref{Qdot-rep} gives
\[
\dot{x}_1 =    \frac{Q_{k,c}(0)^2}{2x_1  Q_{k,c}(-x_1)Q_{k,c}'(x_1)} \to -\frac14
\]
as $c\to \sup J_{k,0}$, since  $x_1\to 0$ and $Q_{k,c}'(x_1) \to \prod_{\ell\neq 1} (-x_\ell)$. The ODE gives
\[
\frac{\dot{x}_1}{x_1} = \frac{1}{c} \frac{x_1^2 -1}{\left(x_1^2/\beta_k(c)\right) - 1}
\]
and hence $x_1(c)^2/\beta_k(c) = 1 + \mc{O}(x_1(c))$ as $c\to \sup J_{k,0}$. The second form of $\beta_k(c)$ in \eqref{betak-two-forms} gives
\[
R_k(c)^2 =\frac{Q_{k,c}(0)^2}{x_1(c)^2 } = \frac{\beta_k(c)}{x_1(c)^2}  \cdot \frac{2Q_{k,c}(1) Q_{k,c}(-1) + cQ_{k,c}(0)^2}{c}
\]

Multiplying by $4^k$, letting $c\to \sup J_{k,0}$ and observing $Q_{k, \sup J_{k,0}}(0)=0$, the lower bound from \eqref{Qpm1-bounds} and  monotonicity of $R_k(c)$ give the lower bound in \eqref{Rk-full-bounds}.
\end{proof}

The Deift-Zhou steepest descent in the next section provides $\sup J_{k,0} \sim   \theta_*^2 (4k)^{-1}$. Hence the bounds in \eqref{Rk-full-bounds}  are uniform in $c$.

\section{Riemann-Hilbert Problem }\label{RHP-section}

 We reformulate the orthogonality as a Riemann-Hilbert problem.  For a general overview we refer to the lecture notes \cite{Dei99}.  The article \cite{KMV04} treats strong asymptotics of weights on $[-1,1]$. Weights with an interior essential singularity were investigated in \cite{ACM16, BMM15}. The corresponding problem of a singularity at an end point goes back to \cite{CIts10}. The present situation is closest to \cite{BMM15} with the added simplification that the model RHP for cosh-Gordon can be imported from \cite{FIKN06} rather than having to compute  the Lax pair.
 
 Let $Y$ be a $2\times2$ matrix whose entries are analytic on $\C\backslash\Gamma_R$, denote by $Y_\pm(z)$, $z\in \Gamma_R$, the non-tangential limit value from the $\pm$ side, where $+$ indicates left and $-$ indicates right of the contour direction, and assume
\begin{align}\label{PA-RHP-boundary}
Y_+(z) = Y_-(z) \begin{pmatrix} 1 & W_{c}(z) \\ 0 & 1 \end{pmatrix}
\end{align}
on $\Gamma_R$. (Recall that $\Gamma_R$ is clockwise oriented.) The normalization at infinity is given by
\begin{align}\label{PA-RHP-infinity}
Y(z) = (I + \mc{O}(z^{-1}))\begin{pmatrix} z^k & 0 \\ 0 & z^{-k}\end{pmatrix}
\end{align}

 It is known that the solution $Y$  of this Riemann-Hilbert problem exists and is uniquely determined. As usual, $\sigma_3$ and $A^{\sigma_3}$ for scalar $A$ are given by
\[ 
\sigma_3 = \begin{pmatrix} 1 & 0 \\ 0 & -1\end{pmatrix}, \qquad A^{\sigma_3} =  \begin{pmatrix} A & 0 \\ 0 & A^{-1} \end{pmatrix}
\]

In terms of the monic orthogonal polynomials $q_{k,c}(z)$, for $z\notin \Gamma_R$ the solution $Y$ has the representation
\[
Y(z) =\begin{pmatrix} q_{k,c}(z)&  \displaystyle  \frac{1}{2\pi i} \displaystyle  \int_{\Gamma_R} \frac{q_{k,c}(t) W_c(t)}{t-z} dt \\ -\frac{2\pi i}{h_{k-1}(c)} q_{k-1,c}(z) &  \displaystyle  -\frac{1}{ h_{k-1}(c)}  \displaystyle  \int_{\Gamma_R} \frac{q_{k-1,c}(t)W_c(t)}{t-z} dt\end{pmatrix}
\]

We set 
\[
c = \frac{d}{k}
\]
 with  $d$ independent of $k$.   

\subsection{Contour deformation.} Let $0<r<1/2$. We define $\Sigma = \Sigma_1\cup...\cup \Sigma_4$ with
\begin{align*}
\Sigma_1 = [-1,-r], \qquad \Sigma_2 = [r,1], \qquad \Sigma_3 = \{|z|=r\} \qquad \Sigma_4 =\{|z|=R\}
\end{align*}
where the line segments are oriented left to right, $\Sigma_3$ is oriented {\it counterclockwise}  and $\Sigma_4$ is oriented {\it clockwise}.

   Define $Z$ on $\C\backslash \Sigma$   by 
\[
Z(z) = \begin{cases}
Y(z)  & \text{ if }|z|<r\text{ or } |z|>R\\
Y(z)  \begin{pmatrix} 1 & -W_{c}(z) \\ 0 & 1 \end{pmatrix}  & \text{ if } r<|z|<R
\end{cases}
\]
and observe that $Z$ is continuous on $\Gamma_R$.  Define
\begin{align}\label{omegac-def}
\omega_c(z) = \sqrt{\frac{1-z}{1+z}}e^{-c/(2z)}  \textrm{ for }z\in \C\backslash\big((-\infty, -1]\cup [1,\infty)\cup\{0\}\big)
\end{align}
so that $(W_c)_+ - (W_c)_- = \omega_c$ on $(-1,1)\backslash\{0\}$. We define $\omega_c^{1/2}$ on the same domain.  The asymptotic of $Z$ at infinity is
\[
Z(z) = (I + \mc{O}(z^{-1})) z^{k\sigma_3}
\]
 
 Define $\phi$ for $z\in \C\backslash[-1,1]$ by
\[
\phi(z) = z + (z^2-1)^{1/2}
\]

Then $\phi(z) = 2z +\mc{O}(z^{-1})$ as $z\to\infty$, it maps $\C\backslash[-1,1]$ to the exterior of the unit disk, is positive for $z>1$, and $\phi_+(x) \phi_-(x) =1$ for $-1<x<1$. Denoting by $\arcsin(z)$ the branch analytic on $\C\backslash ((-\infty,-1]\cup [1,\infty))$ with $\arcsin(0)=0$, the  identities
\[
\phi(z) =  i^{\pm 1}e^{\mp i\arcsin(z)}\qquad \text{ for }\pm \Im z>0
\]
provide the analytic continuations of the boundary values  across  $[-1,1]$. Define $T$ by 
\[
T(z) = 2^{k\sigma_3} Z(z) \varphi(z)^{- k\sigma_3}    
\]
 where $\varphi =\phi$ for $|z|>r$ and branch cut deformed along the semi circle $|z|=r$ in $\Im z<0$. We note that even though $|\varphi_\pm|\neq 1$ on the deformed cut, the identity $\varphi_+\varphi_- =1$ remains true, and hence
\[
\varphi_-^{k\sigma_3} = \varphi_+^{-k\sigma_3}
\]

\subsection{Lens Opening} We define $S$ on $\C\backslash\Sigma$ by  
\[
S = \begin{cases}
T    &\text{ if }  |z|<r\text{ or }|z|>R\\
T\begin{pmatrix} 1 & 0 \\ -\frac{\varphi^{-2k}}{\omega_c} & 1 \end{pmatrix}   & \text{ if }  r<|z|<R\text{ and }\Im z>0, \\
T\begin{pmatrix} 1 & 0 \\ \frac{\varphi^{-2k}}{\omega_c} & 1\end{pmatrix}   & \text{ if }  r<|z|<R\text{ and }\Im z<0.
\end{cases}
\]
and we observe that $S$ extends continuously to $ (-R,-1)\cup (1,R)$ since $\omega_c$ has analytic continuation $-\omega_c$ across $(-R,-1)$ and $(1,R)$ while $\varphi$ is analytic across both segments. The underlying heuristic is a lens opening  
\[
\begin{pmatrix} \varphi_-^{2k} & \omega_c \\ 0 & \varphi_+^{2k} \end{pmatrix} = \begin{pmatrix} 1& 0 \\ \varphi_+^{2k}/\omega_c & 1 \end{pmatrix} \begin{pmatrix} 0 & \omega_c \\ -1/\omega_c & 0\end{pmatrix} \begin{pmatrix} 1 & 0 \\ \varphi_-^{2k}/\omega_c & 1 \end{pmatrix}
\]
on $[-1,-r]$ and on $[r,1]$. The lens contour above $[r,1]$ is then swept to the contour $[R,1] \cup \{R e^{i\theta}: 0\le \theta\le \pi/2\}\cup [ir, iR] \cup \{ re^{i\theta}: 0\le \theta\le \pi/2\}$ and analogously for the other three lens contours, with cancellation on $\pm[ir, iR]$. 

To compute the jumps of $S$, we record the identities
\begin{align*}
S = 2^{k\sigma_3} Y\begin{cases}
\varphi^{-k\sigma_3} &\text{ if }|z|<r,\\
\begin{pmatrix} 1 & W_c \\ 0 & 1 \end{pmatrix} \begin{pmatrix} 1 & 0 \\ -\frac{1}{\omega_c} & 1 \end{pmatrix}\varphi^{-k\sigma_3} &\text{ if }r<|z|<R, \Im z>0\\
\begin{pmatrix} 1 & W_c \\ 0 & 1 \end{pmatrix} \begin{pmatrix} 1 & 0 \\ \frac{1}{\omega_c} & 1 \end{pmatrix} \varphi^{-k\sigma_3} & \text{ if }r<|z|<R, \Im z<0
\end{cases}
\end{align*}

 We denote by $\Sigma_{3,\pm}$ the upper and lower half of $|z|=r$, respectively, and define $\Sigma_{4,\pm}$ analogously.  Since left multiplication by $Y$ affects the jumps only on $\Sigma_{4,\pm}$ where $Y$ is discontinuous,  $S $ satisfies
\begin{enumerate}
\item $S$ is analytic on $\C\backslash\Sigma$
\item $S(z) = I +\mc{O}(z^{-1})$ as $z\to \infty$
\item $S_+ = S_- J$ on $\Sigma$ where
\[
J= \begin{cases}
\varphi^{k\sigma_3} e^{-c/(4z)\sigma_3}  \begin{pmatrix} 1 & 0 \\ \mp \frac{1}{\omega_0}  & 1 \end{pmatrix}  e^{c/(4z)\sigma_3}\varphi^{-k\sigma_3} & \text{ on }\Sigma_{4,\pm}\\
\varphi^{k\sigma_3}  e^{-c/(4z)\sigma_3} \begin{pmatrix} 1 & -W_0 \\ \frac{1}{\omega_0} & \frac{1}{2} \end{pmatrix}e^{c/(4z)\sigma_3} \varphi^{-k\sigma_3} &\text{ on }\Sigma_{3,+}\\
\varphi_-^{k\sigma_3}  e^{-c/(4z)\sigma_3} \begin{pmatrix} 1 & -W_0 \\ -\frac{1}{\omega_0} & \frac{1}{2} \end{pmatrix} e^{c/(4z)\sigma_3}\varphi_+^{-k\sigma_3}& \text{ on }\Sigma_{3,-}\\
 e^{-c/(4z)\sigma_3} \begin{pmatrix} 0 & \omega_0 \\ -\frac{1}{\omega_0} & 0 \end{pmatrix}e^{c/(4z)\sigma_3} & \text{ on }\Sigma_1\cup \Sigma_2
\end{cases}
\]
\item For $z\notin \Sigma$, 
\begin{align*}
S(z)=\mc{O}\begin{pmatrix}|z-1|^{-1/2} & 1 \\ |z-1|^{-1/2} & 1\end{pmatrix}\ (z\to1),\\
S(z)=\mc{O}\begin{pmatrix}1 & |z+1|^{-1/2}\\  1 & |z+1|^{-1/2}\end{pmatrix}\ (z\to-1),
\end{align*}
\end{enumerate}

The growth conditions at $z = \pm1$ imply that $S$ is the  unique solution for the RHP defined by conditions (1)-(4).

\subsection{Global parametrix} Define an RHP using the jump of $S$ on $(-1,1)$, and the asymptotics as $z\to \infty$ and $z\to \pm1$.  From \cite{KMV04} we obtain the form of the solution $N$. Let $a$ be given by
\[
a(z) = \left(\frac{z-1}{z+1}\right)^{1/4}
\]
with branch cut $[-1,1]$ and positive values for $z>1$. The global parametrix  $N$ is
\[
N = \begin{pmatrix} \frac{a+a^{-1}}{2a} &  \frac{(a-a^{-1})a}{2i} \\ -\frac{a-a^{-1}}{2ia} &   \frac{(a+a^{-1})a}{2}\end{pmatrix}.
\]

On $(-1,1)$ the identities  $a_+ a_- = \omega_0$, $a_+/a_- = i$, $a_+^2 = i\omega_0$ and $a_-^2 = -i\omega_0$ are valid, and the boundary values at the origin are
\[
N_\pm(0)= \frac12 \begin{pmatrix} 1 \mp i & i \pm 1  \\ i \mp 1  & 1 \pm i \end{pmatrix}
\]

 \subsection{Model problem}  Let $\delta>0$ with $r<\delta<1/2$. In $|z|<\delta$, the jump $J$ of $S$ can be written as
\[
J= \begin{cases}
\varphi^{k\sigma_3}  e^{-c/(4z)\sigma_3} \omega_0^{\frac12 \sigma_3}  \begin{pmatrix} 1 & -\frac12 \\ 1& \frac{1}{2} \end{pmatrix} \omega_0^{-\frac12 \sigma_3}   e^{c/(4z)\sigma_3} \varphi^{-k\sigma_3} \qquad \text{ on }\Sigma_{3,+}\\
\varphi_-^{k\sigma_3}  e^{-c/(4z)\sigma_3}  \omega_0^{\frac12 \sigma_3}   \begin{pmatrix} 1 & \frac12  \\ -1 & \frac{1}{2} \end{pmatrix}  \omega_0^{-\frac12 \sigma_3}   e^{c/(4z)\sigma_3}\varphi_+^{-k\sigma_3}\qquad  \text{ on }\Sigma_{3,-}\\
 e^{-c/(4z)\sigma_3}  \omega_0^{\frac12 \sigma_3}  \begin{pmatrix} 0 & 1 \\ -1 & 0 \end{pmatrix} \omega_0^{\frac12 \sigma_3}  e^{c/(4z)\sigma_3} \qquad \text{ on }(\Sigma_1\cup \Sigma_2)\cap\{|z|<\delta\}
\end{cases}
\]

Define
\[
B= \begin{pmatrix} 1 & \frac12 \\ 0 & 1 \end{pmatrix}   \begin{pmatrix} 1 & 0 \\ - 1 & 1 \end{pmatrix} = \begin{pmatrix} \frac12 & \frac12 \\ -1 & 1 \end{pmatrix}
\]
and observe with 
\[
\sigma = \begin{pmatrix} 0 & 1 \\ -1 & 0 \end{pmatrix}
\]
that 
\[
\begin{pmatrix} 1 & -\frac12 \\ 0 & 1 \end{pmatrix} \begin{pmatrix} 1 & 0 \\ 1 & 1 \end{pmatrix} \sigma = B
\]

The insertion of $\sigma$ and its inverse  is due to the fact that  $\varphi^{\sigma_3}$ has a jump on its branch cut that interchanges its diagonal. Among all conjugations that undo that switch, the particular matrix used here is also the   jump of the global parametrix at the origin. We observe that with $p = -q = i$ and 
\begin{align}\label{connectionQ}
Q = \frac{1}{\sqrt{1+pq}} \begin{pmatrix} 1 & p \\ -q & 1 \end{pmatrix}
\end{align}
the identity
\[
B =   e^{\frac{\pi i}{4}  \sigma_3} \, \sigma^{-1} 2^{\frac12 \sigma_3} Q  \sigma  \, e^{-\frac{\pi i}{4}  \sigma_3 }
\]
holds. Define $D$ by
\[
D =  \varphi^{-k\sigma_3}  \omega_0^{-\frac12 \sigma_3} e^{c/(4z) \sigma_3}
\]

With this notation, $J$ has the representation
\begin{align}\label{simple-J}
\begin{split}
J = \begin{cases}
 D^{-1} B^{-1} D  &\text{ on }\Sigma_{3,+}\\
 D_-^{-1} \sigma B^{-1} D_+ &\text{ on }\Sigma_{3,-}\\
D_-^{-1} \sigma D_+&\text{ on }\Sigma_1\cup \Sigma_2
\end{cases}
\end{split}
\end{align}
where we used $\sigma =  \varphi_-^{k\sigma_3} \sigma \varphi_+^{-k\sigma_3}$ for the third expression. This allows reading off the model problem from \cite[Chapter 13]{FIKN06}. 

\subsection{Lax pair} Let $\lambda,x  \in \C$ and consider the Lax pair $\Psi_\lambda = A\Psi$, $\Psi_x = U\Psi$ where
\begin{align*}
A &= -\frac{ix^2}{16} \sigma_3 - \frac{ixu_x}{4} \sigma_1 \lambda^{-1} +i(\cos u \sigma_3 - \sin u \sigma_2)\lambda^{-2}\\
U &= -\frac{iu_x}{2} \sigma_1 - \frac{ix}{8} \sigma_3 \lambda
\end{align*}
where $u = u(x)$ satisfies
\[
u_{xx} + \frac{u_x}{x}  +\sin(u) =0
\]

  We let $\Psi_\kappa^{(0)}, \Psi_\kappa^{(\infty)}$ be the fundamental solutions from \cite[(13.1.11)]{FIKN06} 
\begin{align*}
\Psi_\kappa^{(\infty)} &\sim (I + \psi_1^{(\infty)} \lambda^{-1} +\hdots)e^{-\frac{ix^2\lambda }{16}\sigma_3 }\qquad \lambda\to\infty \text{ in }\Omega_\kappa^{(\infty)}\\
\Psi_\kappa^{(0)} &\sim P_0(I +\hdots) e^{-\frac{i}{\lambda} \sigma_3}\qquad \lambda\to0 \text{ in }\Omega_\kappa^{(0)}
\end{align*}
where $\Omega_\kappa^{(0,\infty)}$ are slitted complex planes, and
\[
P_0=\begin{pmatrix} \cos(u/2) & -i \sin(u/2) \\ -i \sin(u/2) & \cos(u/2) \end{pmatrix}
\]

\noindent{\it Remark.} Setting $v = \frac1i(u-\frac\pi2)$ and $\theta = e^{i\pi/4} x$, we have $P_0 = Me^{-i\pi /4 \sigma_3} w^{-\sigma_3} M^{-1}$ with
\begin{align}\label{intro-wtheta}
w = e^{-v(\theta)/2}, \qquad M = \begin{pmatrix} 1 & 1 \\ 1 & -1 \end{pmatrix}
\end{align}

We show below that the critical $\theta$ satisfies $w(\theta)=0$. Since the second column contains $w^{-1}$, the RHP has a singularity for that value, and the small norm estimate needs error bounds that are uniform in $k$ and $|w|$.

\bigskip

Returning to $\Psi$, we fix its remaining parameters by requiring that
\[
\Psi_1^{(\infty)} = \Psi_1^{(0)} Q
\]
with $Q$ given by \eqref{connectionQ}. We observe from \cite[(13.1.31)]{FIKN06} that the Stokes matrices for this connection matrix are the identity, i.e., the $\Psi_\kappa^{(0)}$ are restrictions of a single function analytic in $\C\backslash\{0\}$, and analogously for $\Psi_\kappa^{(\infty)}$ which we will also denote by $\Psi_1^{(0)}$ and $\Psi_1^{(\infty)}$, respectively.

We may assume that $k$ is large with $0<d/k<1/2$. Let $\Lambda_+$ be the component of $\C\backslash\{ \lambda: |\sinh(\frac{d}{2k} \lambda)| = r\}$ containing the origin. Then $\overline{\Lambda_+}$ is a convex body, and we set $\Lambda_-=\C\backslash \overline{\Lambda_+}$.  We define $\Psi_{\textrm{HB}} $  by
\begin{align}\label{PsiDef}
\Psi_{\textrm{HB}} = \begin{cases}
\Psi_1^{(\infty)}& \text{ on }\Lambda_- \\
 \Psi_1^{(0)}2^{-\frac12 \sigma_3}& \text{ on }\Lambda_+\backslash\{0\}
\end{cases}
\end{align}
and we observe that on $\partial \Lambda_+$  oriented counterclockwise
\[
\Psi_-^{-1} \Psi_+ =   \big(\Psi_1^{(\infty)}\big)^{-1} \Psi_1^{(0)}2^{-\frac12 \sigma_3} =  Q^{-1}  2^{-\frac12 \sigma_3}
\]

\subsection{Local parametrix at the origin} Let $\Omega_\pm$ be the set of $z$ to the left and right, respectively, of the path from $-\delta$ to $\delta$ along $\Sigma_1\cup \Sigma_{3,-}\cup \Sigma_2$.  We set $x = 2\sqrt{d} e^{-\pi i /4} $, $\lambda = - 16\xi/x^2$ and $\Phi(\xi) =\Psi_{\textrm{HB}}(\lambda,x)$ with $\Psi_{\textrm{HB}}$ from \eqref{PsiDef}.   Define for $|z|<\delta$ the function  $P$  by
\begin{align}
P(z) = 
\begin{cases}
E(z)   \Phi(k \arcsin(z)) \, \sigma \, e^{-\frac{\pi i}{4}\sigma_3} D(z) &\text{ on }\Omega_+\\
E(z) \Phi(k\arcsin(z)) \,\,e^{\frac{\pi i}{4}\sigma_3}  D(z)&\text{ on }\Omega_-
\end{cases}
\end{align}
where $E$ is given by
\[
E = \begin{cases}
N \omega_0^{\frac12\sigma_3}\sigma^{-1}e^{-\frac{\pi i}{4}\sigma_3} (-I)^{k/2} &\text{ if }\Im z>0,\\
N\omega_0^{\frac12 \sigma_3}e^{-\frac{\pi i}{4}\sigma_3}(-I)^{k/2}&\text{ if }\Im z<0
\end{cases}
\]

Since $N_-^{-1} N_+ = \omega_0^{\frac12 \sigma_3} \sigma \omega_0^{-\frac12 \sigma_3}$ on $[-1,1]$, we obtain that $E$ is analytic on $|z|<\delta$. 

\begin{enumerate}
\item {\bf Jumps.} We  verify with \eqref{simple-J} that $P$ and $S$ have the same jumps in $|z|<\delta$:
\begin{enumerate}
\item On  $\Sigma_{3,+}$
\[
P_-^{-1} P_+ = D^{-1} e^{\frac{\pi i}{4}\sigma_3}  \sigma^{-1} Q^{-1} 2^{-\frac12 \sigma_3} \sigma e^{-\frac{\pi i}{4}\sigma_3}  D = D^{-1} B^{-1} D
\]
\item On $\Sigma_{3,-}$
\[
P_-^{-1} P_+ = D_-^{-1} \left( \sigma e^{\frac{\pi i}{4} \sigma_3} \sigma^{-1} \right) Q^{-1} 2^{-\frac12 \sigma_3} \sigma  e^{-\frac{\pi i}{4}\sigma_3} D_+ = D_-^{-1} \sigma B^{-1}D_+
\]

\item On $\Sigma_1\cup\Sigma_2$
\[
P_-^{-1} P_+ = D^{-1}_- e^{-\frac{\pi i}{4} \sigma_3} \sigma e^{-\frac{\pi i}{4} \sigma_3} D_+ = D_-^{-1} \sigma D_+
\]
\end{enumerate}
 
\item {\bf Asymptotic on $|z|=\delta$ as $k\to\infty$.} We observe
\[
\sigma e^{-\frac{\pi i}{4}\sigma_3} D= e^{\frac{\pi i}{4}\sigma_3} D^{-1}\sigma
\]
and recalling that $k$ is even, for $\pm \Im z>0$ and $|z| =\delta$, 
\[
e^{i k \arcsin(z) \sigma_3} \varphi(z)^{\pm k\sigma_3} =(-I)^{k/2}     
\]

We have for $\Im z>0$, $|z|=\delta$ and $c = d/k$
\begin{align*}
P&\sim E \left(I -\frac{x^2}{16kz} \psi_1^{(\infty)} +\hdots\right) e^{i k \arcsin(z) \sigma_3}  \sigma e^{-\frac{\pi i}{4}\sigma_3} D(z) \\
&= E\left(I -\frac{x^2}{16kz} \psi_1^{(\infty)} +\hdots\right) e^{i k \arcsin(z) \sigma_3} \varphi^{k\sigma_3}  e^{\frac{\pi i}{4}\sigma_3} \omega_0^{\frac12 \sigma_3} e^{-\frac{c}{4z} \sigma_3} \sigma
\end{align*}

For fixed $d$, it would be sufficient to use $I + \mc{O}(1/k)$ for the term in brackets, but in the final section, we need an estimate for a range of $d$. We bound the coefficient of $z^{-1}$. Referring to the variables in \eqref{intro-wtheta}, 
\[
u_x = i v_\theta \frac{d\theta}{dx} 
\]
and $v_\theta = -2w_\theta/w$.  (Below we use  $|w|$ of size $k^{-\frac16}$. The powers  $z^{-j}$ have errors of size $1/|kw|^j$, so they are bounded by the $j=1$ error.) We get
\begin{align*}
P &= E\left(I + \mc{O}\left(\frac{1}{k|w|}\right)\right) (-I)^{k/2}  e^{\frac{\pi i}{4}\sigma_3}\omega_0^{\frac12 \sigma_3} \sigma
\end{align*}
where the $\mc{O}$-constants do not depend on $w$, and for $\Im z<0$, $|z|=\delta$ 
\begin{align*}
P&\sim E \left(I -\frac{x^2}{16kz} \psi_1^{(\infty)} +\hdots\right) e^{i k \arcsin(z) \sigma_3} \varphi^{-k\sigma_3}  e^{\frac{\pi i}{4}\sigma_3} \omega_0^{-\frac12 \sigma_3} e^{\frac{c}{4z} \sigma_3} \\
&= E (I + \mc{O}(1/(k|w|)))  (-I)^{k/2}e^{\frac{\pi i}{4}\sigma_3}\omega_0^{-\frac12 \sigma_3} 
\end{align*}
also with  $\mc{O}$-constants independent of $w$.

\item {\bf Origin analysis.} For $|z|<r$, 
\begin{align*}
P& \sim E P_0 (I + \mc{O}(1/|kw|)) e^{\frac{d}{4\xi}\sigma_3} e^{\frac{\pi i}{4}\sigma_3} D^{-1}\sigma
\end{align*}
hence
\begin{align*}
P&= E P_0(I + \mc{O}(z)) e^{\frac{d}{4k} (1/\arcsin(z) - 1/z)\sigma_3}  e^{\frac{\pi i}{4}\sigma_3} \omega_0^{-\frac12 \sigma_3} \sigma 
\end{align*}
and this remains bounded as $|z|\to 0$.

\end{enumerate}

\subsection{Local parametrices at $z=\pm 1$} Local parametrices $P_{\pm 1}$ in $|z \mp 1|<\delta$ may be constructed as in \cite{KMV04}, with the modification that pre and post factor of the parametrices carry an additional $e^{\pm d/(4kz)\sigma_3}$ to ensure that $S$ and $P_{\pm 1}$ have the same jumps in the  disks.

\subsection{Small norm estimate} Define $\mc{R}$  by 
\[
\mc{R} = \begin{cases}
SN^{-1} & \text{ for }|z|\ge \delta\text{ with }z\notin [-1,1]\cup \Gamma_R\cup \{z:|z\mp 1|<\delta\},\\
S P^{-1}&\text{ for }|z|<\delta\\
S P_{\pm1}^{-1}&\text{ for }|z\mp 1|<\delta
\end{cases}
\]

By construction, $\mc{R}$ has no jumps in $|z|<\delta$ and an isolated singularity at the origin. Since $S$ and $P$ separately remain bounded near the origin, the singularity is removable, and $\mc{R}$ is analytic in $|z|<\delta$. The analogous statement for $|z\pm 1|<\delta$ follows as in \cite{KMV04}. 

The jump contour of $\mc{R}$ consists of the circles $|z|=\delta$, $|z\mp 1|=\delta$, $|z|=R$, and the line segments $\pm[\delta, 1-\delta]$.

On $|z|=R$ the jump of $S$ is $I+\mc{O}(e^{-\mu k})$ for positive $\mu$  while $N$ has no jump. On $|z|=\delta$ oriented counterclockwise and $\Im z>0$,
\begin{align*}
\mc{R}_-^{-1} \mc{R}_+&= N^{-1} P\\
&= N^{-1} E(z)   \Phi(k \arcsin(z)) \, \sigma \, e^{-\frac{\pi i}{4}\sigma_3} D(z) \\
&= N^{-1}E(I + \mc{O}(1/k|w|))   e^{\frac{\pi i}{4}\sigma_3}\omega_0^{\frac12 \sigma_3} \sigma\\
&= I+\mc{O}(1/(k|w|))
\end{align*}

On $|z|=\delta$  and $\Im z<0$,
\begin{align*}
\mc{R}_-^{-1} \mc{R}_+&= N^{-1} P\\
&= N^{-1}E(z) \Phi(k\arcsin(z)) \,\,e^{\frac{\pi i}{4}\sigma_3}  D(z)\\
&= N^{-1}E (I + \mc{O}(1/k))  e^{\frac{\pi i}{4}\sigma_3}\omega_0^{-\frac12 \sigma_3} \\
&= I +\mc{O}(1/(k|w|))
\end{align*}

On $|z\mp 1|=\delta$ the analogous estimates for $\mc{R}$ follow from \cite{KMV04}, while the estimate on $|z|=R$ follows since the lens factors decay exponentially. The estimate $|\omega_{d/k}(z) - \omega_0(z)| = \mc{O}(1/k)$ on real $\delta\le |z|\le 1$ is used to bound the jumps on the line segments. 

It follows from the small norm theorem \cite{Dei99}  that with the variables from \eqref{intro-wtheta}
\[
\mc{R}(z,\theta) = I +\mc{O}(1/|kw(\theta)|)
 \]
with constants independent of $\theta$.

\section{Constant term asymptotics}\label{constant-asymptotics-section}

In this section we prove Theorem \ref{theorem1}. We consider $Y_{11}(\sin(\xi/k))$ where $\xi = k \arcsin(z)$, hence we may assume $|z|<r$.  We have $Y_{11} = Z_{11}$ and 
\begin{align*}
Z &= 2^{-k\sigma_3} T \varphi^{k\sigma_3}\\
&=  2^{-k\sigma_3} (I + \mc{O}(1/(k|w|)))E(z)   \Phi(k \arcsin(z)) \, \sigma \, e^{-\frac{\pi i}{4}\sigma_3} D(z) \varphi^{k\sigma_3}
\end{align*}

Setting $z = \sin(\xi/k)$, we define $F$ by
\[
F(\xi,d) = \lim_{k\to\infty} 2^{k\sigma_3} Z(\sin(\xi/k)) 
\]
and we obtain
\begin{align*}
F(\xi,d) &= E(0)\Phi(\xi)  e^{\frac{\pi i}{4}\sigma_3} e^{-\frac{d}{4\xi}\sigma_3} \sigma\\
&=N_-(0)e^{ -\frac{\pi i}{4}\sigma_3}\Psi_1^{(0)}(-4i\xi/d, 2\sqrt{d} e^{-\pi i /4}  ) 2^{-\frac12 \sigma_3}  e^{\frac{\pi i}{4}\sigma_3}   e^{-\frac{d}{4\xi}\sigma_3} \sigma
\end{align*}

Since $\Psi_1^{(0)}(-4i\xi/d, 2\sqrt{d}e^{-\pi i /4} ) e^{-d/(4\xi)\sigma_3}$ has a removable singularity at the origin, $F$ extends to an entire function in $\xi$, and the asymptotic series at the origin gives its power series. Moreover, the connection formula shows that $F$ is an entire function of exponential type $1$ which is bounded on the real $\xi$-axis.
 
We have
\[
F_{11}(0,d) = \frac{1}{\sqrt{2}}\big( (1+i) \sin(u/2) +(1-i) \cos(u/2)\big) = \frac{1}{\sqrt{2}}(1-i) e^{i u/2}
\]
 
Since $\arg x = -\pi/4$, we set $\theta = e^{\pi i/4} x$, and we compute that $v$ defined by
\[
v = \frac1i \left(u -\frac\pi2\right)
\]
satisfies the cosh-Gordon equation
\begin{align}\label{cosh-Gordon-proof}
v_{\theta\theta} +\frac{1}{\theta} v_\theta = \cosh(v)
\end{align}

In terms of $v$,
\begin{align}\label{Fd0-rep}
F_{11}(0,d) = e^{-v(\theta)/2}
\end{align}
with $\theta = 2\sqrt{d}$.  Since $k$ is even, $g_k$ defined by
\[
g_k(\theta) = (-1)^{k/2} 2^k q_{k,\theta^2/(4k)}(0)
\]
is real valued. Let $\theta_*$ be the smallest positive value where $v(\theta)\to \infty$ as $\theta\to \theta_*-$.
 
\begin{lemma} $w = e^{-v(\theta)/2} $ is strictly decreasing on $(0,\theta_*)$ and extends analytically across $\theta_*$ with $w(\theta_*) =0$ and $w'(\theta_*) = -\frac12$
\end{lemma} 

\begin{proof} Writing the differential equation for $v$ as $(\theta v_\theta)_\theta = \theta \cosh v$, noting that this is $\ge \theta$ and integrating shows that $v_\theta\ge \theta/2$, which gives monotonicity. 
 A direct calculation shows that $w$ satisfies the equation 
\[
ww'' - (w')^2 + \frac{ww'}{\theta} +\frac14 + \frac{w^4}{4} =0
\]
which gives the two statements.
\end{proof}

Hence $w$ has a simple zero at $\theta_*$ from positive to negative values.  Going back to the representation of $Z_{11}$, the error bound picks up an additional factor $1/|w|$ from the second column of $P_0$ (cf.\ \eqref{intro-wtheta}). Hence
\begin{align}\label{gk-uniform-error}
g_k(\theta) = w(\theta)+ \mc{O}\left(k^{-1}|w(\theta)|^{-2}\right)
\end{align}

 Set $s_k = k^{-1/6}$ and consider $g_k(\theta_* \pm s_k)$. At these points, $|w|$ can be bounded below by $c_1 s_k$ with $c_1>0$, hence the error in \eqref{gk-uniform-error} is $\mc{O}(k^{-2/3})$ and it follows that $\pm g_k(\theta_*\pm s_k )>0$. Hence $g_k$ has a zero in $[\theta_*-s_k,\theta_*+s_k]$, and monotonicity of $w$ may be used to show $g_k(\theta)>0$ for $d<\theta_*-s_k$.  We obtain
\[
\lim_{k\to \infty} k C_{2k-2} = \frac{\theta_*^2}{4}.
\]

We consider finally the asymptotic at $d=0$. For $c=0$, we have $q_{k,0}(0) = 2^{-k} H_k(0)$ where $H_k$ is the Chebyshev polynomial of the fourth kind. Hence $g_k(0)=1$. The small norm estimate holds uniformly for $0\le d\le d(\theta_*)/2$, hence $g_k(\theta)\to F_{11}(0,d)$ with uniform error bounds. Hence the limit in $d$ and $k$ can be interchanged and $\lim_{d\to 0+} F_{11}(0,d) = 1$.   It follows that $v$ is the solution of cosh-Gordon that is regular at the origin with $v(0) =0$ and $v'(0)=0$,  completing the proof of Theorem \ref{theorem1}.

\section{Convergence of the extremal polynomials}\label{extremal-convergence-section}

Denote by $f_1^*$ the extremal function for \eqref{PWC-extremal} with $\tau=\pi$ from (notation as in \cite[(1.12)]{LevLub15}). To relate $f_1^*$ to $F_{11}$, we define $G_k$ by
\[
G_k(\xi,d) = \prod_{\ell =2}^k \left(1 - \frac{\sin(\xi/k)}{x_{\ell,k}(d/k)}\right)
\]

Let $d<\theta_*^2/4$. Since $q_{k,c}$ is monic,  
\[
 (-1)^{k/2} q_{k,d/k}(\sin(\xi/k)) = \left(\sin(\xi/k)  - x_{1,k}(d/k)\right) R_k\left(\tfrac{d}{k}\right) G_k(\xi,d)
\]
with $R_k$ from \eqref{Rk-full-bounds}. The relation
\begin{align}\label{F11-to-G}
F_{11}(\xi,d) = \left(\sin(\xi/k)  - x_{1,k}(d/k)\right) \left(2^k R_k\left(\tfrac{d}{k}\right)\right) G_k(\xi,d) +\mc{O}\left(\tfrac{1}{k |w(2\sqrt{d})|}\right)
\end{align}
holds with an error that is uniform in $\xi$ in compact subsets of $\C$ (although not uniform in $d$). Separation of zeros gives $k x_{1,k}(d/k) \in [0,\pi/4]$, and all other $x_{\ell,k}$ are not in that interval. The function $F_{11}$ is not identically zero by \eqref{Fd0-rep}. Choosing a simple closed contour circling $[0,\pi/4]$ that contains no zeros of $F_{11}$, the Hurwitz theorem implies that $F_{11}$ has exactly one zero inside that contour. Since \eqref{Fd0-rep} gives also $F_{11}(0,d)\neq 0$,
\[
\lim_{k\to \infty}k  x_{1,k}(d/k) = \xi_1(d)
\]
where $\xi_1(d)$ is the least positive zero of $F_{11}(\xi,d)$. Define $F$ by
\[
F(\xi,d) = \frac{F_{11}(\xi,d)}{\xi - \xi_1(d)}
\]

For the following calculations, we define
\[
d_k^* = k\sup J_{k,0}
\] 
and observe that $x_1(d_k^*/k) =0$, $Q_{k, d_k^*/k}(0)=0$ and $d_k^* \to \theta_*^2/4$ as $k\to\infty$. 

\begin{lemma} For $d<\theta_*^2/4$,
\[
\lim_{k\to\infty} G_k(\xi,d) = \frac{F(\xi,d)}{F(0,d)}
\]
uniformly in compact subsets of $\C$.  
\end{lemma}

\begin{proof}  We have
\begin{align*}
\log G_k(\xi,d)  =  \sum_{\ell= 2}^{k} \log\left( 1 - \frac{\sin(\xi/k)}{ x_{\ell,k}(d/k)}\right)
\end{align*}

Since $x_{\ell}$ alternate in sign, 
\[
\sum_{\ell=2}^k \frac{1}{k x_{\ell,k}} = \sum_{n=1}^{k/2} \frac{(2k)  \gamma_{2n} }{(2k)(2k-1) x_{(2n-1)} x_{(2n)}}
\]
with the gaps $\gamma_{2n}(d/k)$, and the numerator is uniformly bounded by \eqref{gap-bound}, and hence the sum as well.  The sum of the remaining terms from the expansion of the logarithm are uniformly bounded by \eqref{square-zero-bound}. Since $G_k(0,d)=1$, the claim follows.
\end{proof}

We emphasize that the bounds for $G_k(\xi, d)$ in this proof are uniform for $d\le d_k^*$, but the definition of $F_{11}$ does not extend to $d= \theta_*^2/4$, and the error in \eqref{F11-to-G} is not uniform in $d_k^*$ due to the term $1/w(\theta)$.

\begin{lemma} We have for $|\xi|\le \pi/4$,
\begin{align}\label{partial-bound}
\int_d^{d_k^*} \left|\partial_s G_k(\xi,s) \right| ds = C|d-d_k^*|
\end{align}
with a constant independent of $k$. 
\end{lemma}

\begin{proof} As mentioned above, the bounds for $G_k(\xi,d)$ hold uniformly for $d\le d_k^*$. We have 
\[
\partial_d \log G_k(\xi,d) = \sum_{\ell=2}^k  \frac{\partial_d [x_{\ell}]}{x_{\ell}(x_{\ell} - \sin(\xi/k))} 
\]
and this is uniformly bounded by \eqref{xdot-bound} for $|x|\le \pi/4$ and $d\le d_k^*$. 
\end{proof}

Define $I_k$ by 
\[
I_k(\xi) = G_k(\xi, d_k^*)
\]

\begin{lemma}\label{normal-family-lemma} \begin{enumerate}
\item $I_k$ is a normal family on any compact subset of $\C$,
\item $I_k$ is a Cauchy sequence on $|\xi|\le \pi/4$.
\end{enumerate}
\end{lemma}

\begin{proof}  For the Cauchy sequence, estimate for $m\ge k$,
\begin{align*}
|I_k - I_m|\le |G_k(\xi, d_k^*) - G_m(\xi, d_k^*)| + |G_m(\xi, d_k^*) - G_m(\xi, d_m^*)|
\end{align*}

The first term is small since $G_k$ at fixed $d$ is Cauchy, the second is small by \eqref{partial-bound}.
\end{proof}

As a consequence,
\[
I = \lim_{k\to\infty} I_k
\]
 is entire in $\xi$.

\begin{theorem}\label{theorem3} \begin{enumerate}
\item $I$ is entire of exponential type $1$
\item $\xi\mapsto \xi I(\xi)$ is bounded on the real line
\item Uniformly on compact subsets of $\C$,
\begin{align}\label{Fxid-conv-toI}
\lim_{d\to \theta_*^2/4} \frac{F(\xi, d)}{F(0,d)} = I(\xi)
\end{align}
\end{enumerate}
\end{theorem}

\begin{proof} From $|x_{\ell,k} |\ge |x_{\ell,k}(0) - x_{1,k}(0)|$ and $|x_{\ell,k}(0)|\to (2m-1)\pi/(4k)$, the zeros $\xi_\ell$  of $I$ have the property that $|\xi_\ell|$ are asymptotically of size $(m-1)\pi/2$ . Since $x_\ell$ alternate  in sign, $I$ has exponential type $1$ and $\xi I(\xi)$ is bounded on the real line. 

We now estimate
\begin{align*}
\left|\frac{F(\xi,d)}{F(0,d)} - I(\xi)\right|\le \left|\frac{F(\xi,d)}{F(0,d)} - G_k(\xi,d)\right| + |G_k(\xi,d) - G_k(\xi, d_k^*)| + |I_k(\xi) - I(\xi)|
\end{align*}
and apply \eqref{partial-bound} and Lemma \ref{normal-family-lemma}.
\end{proof}
  
\begin{corollary} $I(\xi) I(-\xi) = f_1^*(2\xi)$. 
\end{corollary}

\begin{proof} The sign changes of $H_{2k}$ in \eqref{bell} are located at the zeros of $q_{k,c} \gamma_{k,c}$, and hence the extremal polynomial $P_{2k-2,*}$ for \eqref{Cn-extremal} with \eqref{poly-signs} for $n = 2k-2$ satisfies
\begin{align}\label{sign-change-match}
x^2 P_{2k-2,*}(x) = \frac{q_{k,C_{2k-2}}(x) \gamma_{k, C_{2k-2}}(x)}{ q_{k,C_{2k-2}}'(0) \gamma_{k, C_{2k-2}}'(0) }  
\end{align}

Setting $x = \sin(\xi/(2k))$, it follows from \cite{LevLub15} (with their notation)
\begin{align}\label{Ptof1star}
\lim_{k\to \infty} P_{2k-2,*}(\sin(\xi/(2k)) =  f^*_1(\xi/\pi)
\end{align}
uniformly on compact subsets of $\C$.  

 The claim follows now from $I_k = \lim G_k(\xi, d_k^*)$, integrability of $I(\xi) I(-\xi)$, and $I(0)=1$.
\end{proof}

\section{Associated Heun Differential Equations}\label{Heun-ODE-section}

In this section we prove Theorem \ref{theorem2}. Recall that 
\[
M =\begin{pmatrix} 1 & 1 \\ 1 & -1 \end{pmatrix}
\]

Define $\widehat{\Psi} = M^{-1}\Psi$, $\widehat{A} = M^{-1}AM$ and $\widehat{U} = M^{-1} UM$ so that $\Psi_\lambda = A\Psi$, $\Psi_x = U\Psi$ become $\widehat{\Psi}_\lambda = \widehat{A}\widehat{\Psi}$ and $\widehat{\Psi}_x = \widehat{U}\widehat{\Psi} $. Then
\[
\frac{1}{\sqrt{2}}\begin{pmatrix} 1 & -1 \\ 1 & 1 \end{pmatrix} \Psi = \sqrt{2} \sigma_1 \widehat{\Psi}
\]
and hence
\[
F = \sqrt{2}\sigma_1 \widehat{\Psi} 2^{-\sigma_3/2} e^{i\pi/4 \sigma_3}e^{-d/(4\xi) \sigma_3}\sigma
\]

This gives
\[
F_{11}(\xi, d) = -2 \widehat{\Psi}_{22}(\xi,d) e^{i\pi/4 } e^{d/(4\xi) }
\]

Rewriting in terms of $\theta$, $\xi$, and $v$ from $x$, $\lambda$, and $u$, and observing $d\lambda/d\xi = -16i/\theta^2$, 
\begin{align}\label{Ahat-xid}
\widehat{A} &= i\sigma_1 + \frac{\theta v_\theta}{4} \sigma_3 \xi^{-1} + \frac{i\theta^2}{16} ( \sinh(v) \sigma_1 +\cosh(v)i\sigma_2 )\xi^{-2} 
\end{align}
 
We review how Heun differential equations for a Painlev\' e equation are obtained (e.g., \cite[Section 5]{XXZ21}). Write $\hat{A} = (a_{ij})$, $i,j = 1,2$, and consider the first column $\binom{y_1}{y_2}$ of $\hat{\Psi}$. Solve $(y_1)_\xi  = a_{11} y_1 + a_{12} y_2$ for $y_2$, differentiate, and plug into the second row $(y_2)_\xi = a_{21} y_1 + a_{22} y_2$. From \eqref{Ahat-xid}
\[
\widehat{A} =\begin{pmatrix} \frac{\theta v_\theta}{4\xi} &  i\left(1 + \frac{\theta^2 e^v}{16\xi^2}\right) \\  i\left( 1 - \frac{\theta^2 e^{-v}}{16\xi^2}\right) &  -\frac{\theta v_\theta}{4\xi}\end{pmatrix}
\]

Denoting by $'$ differentiating in $\xi$,
\[
y_1'' - \left( a_{11} + a_{22} +\frac{a_{12}'}{a_{12}}\right) y_1' + \left( a_{11} a_{22} - a_{12} a_{21} - a_{11}' + a_{11} \frac{a_{12}'}{a_{12}}\right) y_1 =0
\]

We have $\Tr \widehat{A} =0$ and 
\[
\det \widehat{A} = 1 +\frac{\theta^2 \sinh(v)}{8\xi^2} -\frac{\theta^2 v_\theta^2}{16\xi^2} -\frac{\theta^4}{256 \xi^4}
\]

 The terms $a_{12}'/a_{12}$ introduce unwanted zeros at singularities of $v$ from
\[
\frac{a_{12}'}{a_{12}} = -\frac{2 e^{v(\theta)}}{\xi(e^{v(\theta)} - \frac{\theta^2}{16} \xi^2)} = - \frac{2}{\xi} + \frac{1}{\xi - \frac{4}{ \theta} e^{v(\theta)/2}} +\frac{1}{\xi + \frac{4}{ \theta}e^{v(\theta)/2}}
\]

 The Heun equations select only those parameters $\theta$ for which the unwanted poles/zeros also land at either zero or infinity. The smallest positive value when this occurs is $\theta \to \theta_*$ where $e^{v(\theta)/2} \to \infty$.  The main calculation involves the coefficient function of $y/\xi^2$. We note that for $\theta \neq \theta_*$, it is given by
\[
H(\theta) = \frac{\theta^2 \sinh(v)}{8}  - \frac{\theta^2 v_\theta^2}{16} -\frac{\theta v_\theta}{4}
\]

Inserting the expansion from \eqref{localv-attheta*}, a (slightly tedious) calculation shows that the coefficients of the negative powers of $\theta - \theta_*$ are zero, and the constant coefficient evaluates to $\frac{3}{4} \theta_*^2 c_2 -\frac{3}{16} $, giving the first Heun equation for row $1$ of $\widehat{A}$ 
\[
y'' + \frac{2}{\xi} y' + \left(1 + \frac{1}{\xi^2} \left(\frac{3}{4} \theta_*^2 c_2 - \frac{3}{16} \right)  -\frac{\theta_*^4}{256 \xi^4}\right) y =0
\]

Comparing with the corresponding equation from \cite[(3.4)]{BORS26-2}, the function $g =e^{1/(4\mathscr{C} z)} \Phi$ satisfies
\[
g'' + \frac{2}{z} g' +\left( \frac{\pi^2}{4} + \frac{L_\tau(1)}{2\mathscr{C}}\frac{1}{z^2}  -\frac{1}{16 \mathscr{C}^2} \frac{1}{z^4} \right) g =0
\]

Since $g'' + (2/z) g'$ is invariant under $z\mapsto cz$ followed by multiplication with $c^2$, comparison of coefficients shows  that
\[
1=\frac{\pi^2c^2}{4}, \qquad \frac{\theta_*^4}{256 } =\frac{1}{16 \mathscr{C}^2 c^2}  ,  \qquad \frac{L_\tau(1)}{2\mathscr{C}} = \frac{3}{4} \theta_*^2 c_2 -\frac{3}{16} 
\]
with $c_2$ from \eqref{localv-attheta*}. Substituting the first into the second gives 
\[
\frac{\pi}{4\mathscr{C}} = \frac{\theta_*^2}{8}
\]
 and plugging the second identity into the third coefficient equation is the conjectured identity between $c_2, \mathscr{C}$, and $L_\tau(1)$. 

Since $F_{11}$ is a coefficient of the second row of $\widehat{\Psi}$, for the relation with $\Phi$ from \cite{BORS26-2}, we need the Heun equation that comes from eliminating through the second row from $\widehat{A}$.  When working out $a_{21}'/a_{21}$, the Heun equations for the second row of $\widehat{A}$ are obtained by choosing $\theta$ that move  unwanted singularities to the origin instead of infinity. The analogous calculations for $\theta = \theta_*$ give
\begin{align}\label{u-ODE}
u '' + \left( \frac{\pi^2}{4} -\left(  \frac{3}{16} -\frac{3}{4} \theta_*^2 c_2\right)\frac{1}{z^2}  -\frac{1}{16 \mathscr{C}^2} \frac{1}{z^4} \right) u =0
\end{align}

The identification with \cite{BORS26-2} uses now $u = z g$, and the parameter identifications are unchanged. 

\begin{proposition} $u = e^{-\theta_*^2/(16 \xi)} \xi I(\xi)$ solves \eqref{u-ODE}
\end{proposition}

\begin{proof} We note that $\widehat{\Psi}_{22}$ solves the ODE coming from the second row of $\widehat{A}$, hence so does
\[
\xi\mapsto (\xi - \xi_1(d)) \frac{F(\xi,d)}{F(0,d)} 
\]

Since $\xi_1(\theta_*^2/4)=0$, \eqref{Fxid-conv-toI} gives the claim.
\end{proof}

We note that the origin zero of $F_{11}$ is consistent with the invisible origin zero mentioned in the remark after \cite[Theorem 1.2]{BORS26-2}.

\bibliographystyle{amsplain}
\bibliography{references-new}

\end{document}